\documentclass[12pt,a4paper]{amsart}

\usepackage[includeheadfoot,
            marginratio={1:1,2:3}, 
            width=412pt, 
            height=688pt,]{geometry}

\usepackage{amsmath}
\usepackage{amsfonts}
\usepackage{amssymb}
\usepackage{graphicx}
\usepackage{ulem}
\usepackage{amsthm}
\usepackage{tikz}
\usetikzlibrary{decorations.markings}
\usetikzlibrary{shapes}
\usetikzlibrary{arrows.meta}
\usetikzlibrary{positioning}
\usetikzlibrary{calc}
\usepackage{mathtools}
\usetikzlibrary{positioning}
\usetikzlibrary{cd}
\usepackage{nicematrix}
\usepackage{array}
\usepackage{float}
\usepackage{xcolor}
\usepackage{enumitem}
\usepackage{BOONDOX-calo}
\usepackage{caption}
\usepackage{hyperref}
\usepackage{txfonts}
\usepackage{bbm}
\usepackage{caption}
\usepackage{framed}

\makeatletter
\renewcommand\section{\@startsection{section}{1}%
  \z@{.7\linespacing\@plus\linespacing}{.5\linespacing}%
 {\normalfont\bfseries\centering}}
\makeatother
\tikzset{%
    symbol/.style={%
        draw=none,
        every to/.append style={%
            edge node={node [sloped, allow upside down, auto=false]{$#1$}}}
    }
}

\tikzset{->-/.style={decoration={
  markings,
  mark=at position .5 with {\arrow{>}}},postaction={decorate}}}
  
  \tikzset{mid/.style 2 args={
        decoration={markings,
            mark= at position #2 with {\arrow{{#1}[scale=1.5]}} ,
        },
        postaction={decorate}
    },
mid/.default={>}{0.5}
}
  \tikzset{dot/.style={circle, fill,inner sep=0pt, minimum size=5pt}}
\tikzset{
  no line/.style={draw=none,
    commutative diagrams/every label/.append style={/tikz/auto=false}},
  from/.style args={#1 to #2}{to path={(#1)--(#2)\tikztonodes}}
  }

\newcommand{\cylinder}[1]{\scalebox{0.5}{\begin{tikzpicture}[baseline=0cm,decoration={markings, 
			mark= at position 0.5 with {\arrow{>[scale=1.5]}}}
		] 
		\draw (0,0) circle (5.5cm);
		\draw (0,0) circle (0.5cm);
		{#1}
	\end{tikzpicture}}}

\newcommand{\ld}[1]{{}^{\vee}{#1}}

\def\point#1(#2,#3)(#4,#5)
{\node[circle,draw,black,fill,scale=0.25] (#1) at (#2,#3) {};
	\node[font=\Large,anchor = #5] at (#2,#3) {#4};
}

\def\line(#1,#2)(#3,#4,#5)(#6,#7) 
{\draw[postaction={decorate}] (#1) to [out = #3, in = #4, looseness = #5] node[font=\Large,anchor = #7] {#6} (#2);
}

\def\rect(#1)(#2,#3) 
	{\draw[dashed] (#1) ++(-0.5*#2,-0.5*#3) -- +(#2,0) -- ++(#2,#3) -- ++(-#2,0)  -- cycle;
}

\newcommand{\eq}[1]{\begin{equation}
                     \begin{aligned} #1 \end{aligned}
                     \end{equation}}

\newcommand{\ov}[1]{\overline{#1}}
\newcommand{\ul}[1]{\underline{#1}}

\DeclareMathAlphabet{\mathpzc}{OT1}{pzc}{m}{it}

\newcommand{\Kbb}{\mathbb{K}}

\newcommand{\Rbb}{\mathbb{R}}
\newcommand{\Zbb}{\mathbb{Z}}

\newcommand{\onebb}{\mathbbm{1}}

\newcommand{\Graphrm}{\mathrm{Graph}}
\newcommand{\VGraphrm}{\mathrm{VGraph}}

\newcommand{\NGraphrm}{\mathrm{NGraph}}

\newcommand{\SNrm}{\mathrm{SN}}

\newcommand{\TVrm}{\mathrm{TV}}

\newcommand{\coevrm}{\mathrm{coev}}
\newcommand{\codomrm}{\mathrm{codom}}
\newcommand{\drm}{\mathrm{d}}
\newcommand{\domrm}{\mathrm{dom}}
\newcommand{\evrm}{\mathrm{ev}}
\newcommand{\idrm}{\mathrm{id}}

\newcommand{\prrm}{\mathrm{pr}}

\newcommand{\spanrm}{\mathrm{span}}

\newcommand{\Acal}{\mathcal{A}}

\newcommand{\Bcal}{\mathcal{B}}
\newcommand{\Ccal}{\mathcal{C}}
\newcommand{\Dcal}{\mathcal{D}}
\newcommand{\Ecal}{\mathcal{E}}
\newcommand{\Fcal}{\mathcal{F}}

\newcommand{\Csf}{\mathsf{C}}

\newcommand{\Zsf}{\mathsf{Z}}

\newcommand{\Cobsf}{\mathsf{Cob}}
\newcommand{\Cylsf}{\mathsf{Cyl}}

\newcommand{\Funsf}{\mathsf{Fun}}

\newcommand{\Lexsf}{\mathsf{Lex}}

\newcommand{\Modsf}{\mathsf{Mod}}
\newcommand{\Natsf}{\mathsf{Nat}}

\newcommand{\Vectsf}{\mathsf{Vect}}

\newcommand{\lbr}{\left\lbrace}
\newcommand{\rbr}{\right\rbrace}
\newcommand{\p}{\partial}

\newcommand{\TFG}{{}_{F}{T_G}}
\newcommand{\ZFG}{{}_{F}{\Zsf_G(\Ccal)}}

\renewcommand{\hom}{\mathrm{Hom}}

\allowdisplaybreaks[2]
\numberwithin{equation}{section}

\theoremstyle{definition}
\newtheorem{defn}{Definition}[section]
\theoremstyle{plain}
\newtheorem{theo}[defn]{Theorem}
\theoremstyle{plain}
\newtheorem{lem}[defn]{Lemma}
\theoremstyle{remark}
\newtheorem{rem}[defn]{Remark}
\theoremstyle{remark}

\theoremstyle{plain}
\newtheorem{Cor}[defn]{Corollary}
\theoremstyle{plain}
\newtheorem{prop}[defn]{Proposition}

\theoremstyle{plain}

\theoremstyle{plain}

\theoremstyle{definition}

\begin{document}

\begin{flushright}
  {\tiny
   {\sf ZMP-HH/23-2}\\
  {\sf Hamburger$\;$Beitr\"age$\;$zur$\;$Mathematik$\;$Nr.$\;$938}\\[2mm]
     February 2023
  
  }
  \end{flushright}

  \title{Twisted Drinfeld Centers and Framed String-Nets}

\vspace{2cm}

\vspace{0.4cm}

\author{
Hannes Kn\"otzele, Christoph Schweigert, Matthias Traube}

\address{Hannes Kn\"otzele: Universit\"at Hamburg \\ 
   Bundesstra\ss e 55, 20146 Hamburg } 
\email{hannes.knoetzele@uni-hamburg.de}
\address{Christoph Schweigert: Universit\"at Hamburg \\ 
   Bundesstra\ss e 55, 20146 Hamburg } 
\email{christoph.schweigert@uni-hamburg.de}
\address{Matthias Traube: Universit\"at Hamburg \\ 
   Bundesstra\ss e 55, 20146 Hamburg } 
\email{matze.traube@gmail.com}
\begin{abstract}
We discuss a string-net construction on $2$-framed surfaces, taking as algebraic input a finite, rigid tensor category, which is assumed to be neither pivotal nor semi-simple. It is shown that circle categories of our framed string-net construction essentially compute Drinfeld centers twisted by powers of the double dual functor.
\end{abstract}

\maketitle

\tableofcontents


\section{Introduction}

Over the last few decades, topological field theories proved to be a very fruitful research area relating concepts from topology, categorical algebra and mathematical physics. A topological field theory (TFT) in $n$ dimensions with values in a symmetric monoidal category $\Ccal$ is a symmetric monoidal functor $\Fcal:\Cobsf^n\rightarrow \Ccal$, where $\Cobsf^n$ is a symmetric monoidal
category with closed $(n-1)$-dimensional topological manifolds as objects; morphisms are given by $n$-dimensional cobordisms. The symmetric monoidal product on $\Cobsf^n$ is given by disjoint union of manifolds. One can consider various tangential structures on objects and morphisms of $\Cobsf^n$, in particular an orientation or an $n$-framing. Then one speaks of oriented or framed TFTs, respectively. 
Most explicitly constructed examples of TFTs are oriented low-dimensional TFTs,
in dimensions $2$ and $3$.  Among the best-known examples of these are the 
Reshetkhin-Turaev \cite{reshetikhin1991invariants} and Turaev-Viro \cite{turaev1992state} TFTs, which are three-dimensional oriented TFTs with values in the category of finite-dimensional $\Kbb$-vector spaces $\Vectsf_\Kbb$, where $\Kbb$ is an algebraically closed field of characteristic zero. The Reshetikhin-Turaev TFT is based on link invariants derived from a modular tensor category, whereas the Turaev-Viro TFT is a state sum construction using a spherical fusion category (see e.g. \cite{turaev2016quantum} for a textbook account of both).

On the other hand, in structural investigations, the case of framed topological
field theories is a natural starting point. Indeed, the
cobordism hypothesis \cite{baez1995higher} is best understood \cite{Lurie}
starting from a suitable category of \textit{framed} cobordisms. In this spirit,
the construction of \cite{douglas2020dualizable} gives explicit categories
associated to framed circles by a $2$-dimensional TFT.

In this article, we address framed theories from the point of view of string-net
constructions. The string-net construction originally emerged in physics \cite{Levin:2004mi}; see however also \cite{walk4} for an early discussion.
A mathematical construction for string-nets that assigns vector spaces to oriented
$2$-manifolds appeared in \cite{kirillov2011string}.
The oriented string-net construction takes as input a spherical fusion category $\Ccal$ and produces for any $2$-dimensional oriented manifold $\Sigma$, possibly with boundary, a finite-dimensional $\Kbb$-vector space $\SNrm^\Ccal(\Sigma)$ with a geometric
action of a mapping class group. Moreover, in \cite{kirillov2011string} it was shown that there is an isomorphism of vector spaces $\SNrm^\Ccal(\Sigma)\simeq \TVrm_\Ccal(\Sigma)$ between the oriented string-net space and the state space of the Turaev-Viro TFT. Since then, string-nets have been used to construct correlators 
in RCFTs \cite{traube2022cardy, fuchs2021string, Schweigert:2019zwt} and have been extended to  non-spherical pivotal fusion categories \cite{runkel2020string} as input data and to manifolds with $G$-bundles
 \cite{Gstring}. 

In this paper, we present a string-net construction on $2$-framed $2$-manifolds, 
see section  \ref{framed construction} for the definition. Working with framed 
rather than with oriented $2$-manifolds means that we have more structure on
the geometric side; as a consequence,  our string-net construction needs as
an algebraic input datum only a  tensor category $\Ccal$, which needs to be
neither semi-simple nor pivotal. The framed string-net space is constructed in 
terms of $\Ccal$-colored oriented graphs, which have to be compatible with the $2$-framing: A $2$-framed two-dimensional manifold $\Sigma$ has two nowhere-vanishing and linearly independent vector fields $X$, $Y$. We only allow oriented
graphs $\Gamma\subset \Sigma$ whose edges are at no point tangent to the
$X$-vector field. This is a globalization to $2$-framed surfaces of the graphical calculus for 
tensor categories in the plane given in \cite{JOYAL199155}, where the $x$-axis 
and $y$-axis of the plane have very different roles and graphs are required to be
\textit{progressive}, i.e. they are not allowed to have tangent vectors pointing 
in the $x$-direction. 

We put the framed string-net construction to the test by computing circle 
categories $\Cylsf(\Csf_n,\Ccal)$ for $\Ccal$ a finite tensor category
that our construction associates to framed
circles. Such circles are classified by an integer $n\in \Zbb$ that counts how often
 the $2$-framing rotates around the circle (see figure \ref{different framings}). 
 In view of the results in \cite{douglas2020dualizable},
 we expect that these circle categories are related to Drinfeld centers twisted 
 by powers of the double dual functor. In fact,
 twisted Drinfeld centers $\ZFG$ can be defined for any pair of strong-monoidal
 functors $F,G:\Ccal\rightarrow \Ccal$: the objects of $\ZFG$ are pairs
 $(c,\gamma_{\bullet,c})$ consisting of an object $c\in \Ccal$ together with 
 a half-braiding $\gamma_{c,x}: F(c)\otimes x \xrightarrow{\simeq} x \otimes G(c)$.

To identify the circle category for the cylinder $\Csf_n$ with a twisted 
Drinfeld center, we use that the twisted Drinfeld center $\ZFG$ is equivalent 
to the category of modules for the twisted central monad $\TFG$ on $\Ccal$. 
We show in Theorem \ref{main theo} that the string-net construction  gives 
us the Kleisli category $\Ccal_{T_n}$ of a specific monad $T_n$  where the 
twisting is by a power of the bidual functor (which is monoidal): 
\eq{
\Cylsf(\Csf_n,\Ccal)\simeq \Ccal_{T_n}   \,\,. 
}
In Theorem \ref{main theo flat} we show that the twisted Drinfeld center
itself can be recovered, as a linear category by taking presheaves on the Kleisli
category for which the  pullback to a presheaf on $\Ccal$ is representable:
\eq{
\mathrm{PSh}_\Ccal(\Cylsf_n)\simeq \Zsf_n(\Ccal)
\label{eq: main result}
}
where $\Zsf_n(\Ccal)$ is the Drinfeld center twisted by the appropriate power 
of the double dual functor depending on $n$, cf. equation \eqref{Zn definition}. 
This allows us to recover twisted Drinfeld
centers from framed string-nets. The comparison with \cite[Corollary~3.2.3]{douglas2020dualizable}
shows complete coincidence. This provides a way to obtain twisted Drinfeld
centers in the spirit of planar algebras \cite{JonesPlanar}; 
they are closely related to tube algebras which can be
formulated as the annular category \cite{Jones} of a planar algebra.

\medskip

This paper is organized as follows. In two preliminary sections, we recall
in section \ref{sec: finite}  some facts 
and notation about finite tensor categories and in section
\ref{sec: DrinMod} about twisted Drinfeld centers and monads.
In this section, we show in particular in Proposition \ref{T-mod is pullback} 
how to obtain the Eilenberg-Moore category of a monad in terms of presheaves
on the Kleisli category whose pullback is representable. While this statement
is known in the literature, in particular  in a general context, we include
the proof for the benefit of the reader.

In section \ref{Graphical calculus for Finite Tensor Categories} we recall the 
graphical calculus of progressive graphs for monoidal categories that has been
introduced in \cite{JOYAL199155}. 
In section \ref{framed construction}, we first show in subsection
\ref{Locally Progressive Graphs} how to globalize the graphical calculus from section \ref{Graphical calculus for Finite Tensor Categories} to $2$-framed surfaces. 
This allows us to define in subsection \ref{Framed String-Net Spaces}  string-net spaces on $2$-framed surfaces, see in particular Definition \ref{framed string-net space}. 

Section \ref{cylinder section} is devoted to the study of circle categories:
in subsection \ref{$2$-Framings of the Circle and Framed Cylinders} we very briefly
discuss framings of cylinders, before we define framed circle categories in section
\ref{Circle Categories} and show in Theorem \ref{main theo} that the 
circle categories
are equivalent to Kleisli categories. Finally, Theorem \ref{main theo flat} 
in section \ref{Cylinder Kleisli} contains the main result (\ref{eq: main result}) and the extension to arbitrary framings
in Remark  \ref{extension to all n remark}. 

\vspace*{0.5cm}
\textbf{Acknowledgment:} The authors thank Gustavo Jasso, Ying Hong Tham and 
Yang Yang for useful discussions. CS and MT are supported by the
Deutsche Forschungsgemeinschaft (DFG, German Research Foundation) under SCHW1162/6-
1; CS is also supported by the DFG under Germany’s Excellence Strategy - EXC 2121 "Quantum
Universe" - 390833306.  HK acknowledges support by DFG under 460925688 (in the Emmy-Noether group of Sven M\"oller).

\section{Recollections on Finite Tensor Categories}
\label{sec: finite}

In this section, we recall some facts about finite tensor categories and at the same
time fix notation. Proofs and more detailed information can be found in e.g. \cite{egno,mac2013categories,kerler2001non}.

Throughout this paper, $\Kbb$ will be an algebraically closed field of characteristic zero. All monoidal categories will be assumed to be strict.

\subsection{Rigid Monoidal Categories}

An abelian monoidal category $(\Ccal, \otimes, \onebb)$ is \textit{$\Kbb$-linear} if it is enriched in $\Vectsf_\Kbb$ and if $\otimes:\Ccal\times\Ccal\rightarrow \Ccal$ is a bilinear functor. A \textit{linear functor} between $\Kbb$-linear categories is an additive functor, i.e. linear on $\hom$-spaces. For $\Kbb$-linear categories $\Dcal$, $\Bcal$, we denote the category of linear functors from $\Dcal$ to $\Bcal$ by $\Funsf_\Kbb(\Dcal,\Bcal)$. For a category $\Ccal$, we denote $\Ccal^{opp}$ for the opposite category, i.e. $\Ccal^{opp}$ has the same objects as $\Ccal$ and $\hom_{\Ccal^{opp}}(x,y)=\hom_\Ccal(y,x)$. For a monoidal category $(\Ccal,\otimes,\onebb)$, its \textit{opposite monoidal category} $\Ccal^{rev}\coloneqq \left(\Ccal^{opp},\otimes^{opp},\onebb\right)$ is the opposite category $\Ccal^{opp}$ endowed with the monoidal structure $x\otimes^{opp} y\coloneqq y \otimes x$ for $x,y\in \Ccal^{opp}$.

A monoidal category $\Ccal$ has \textit{left duals} if for every object
$x \in \Ccal$, there exists an object $\prescript{\vee}{}{x}\in \Ccal$,
called the \textit{left dual object} of $x$, together with a \textit{left coevaluation} $\coevrm_x: \onebb\rightarrow \prescript{\vee}{}{x}\otimes x$ and left evaluation $\evrm_x: x \otimes \prescript{\vee}{}{x} \rightarrow \onebb$ satisfying the usual two zigzag relations.
Similarly, $\Ccal$ has \textit{right duals} if for $x \in \Ccal$, there exists
an object  $x^\vee \in \Ccal$, called the \textit{right dual object}, together with
a \textit{right coevaluation} morphism $\widetilde{\coevrm}_x : \onebb\rightarrow x \otimes x^\vee$ and an \textit{evaluation} morphism $\widetilde{\evrm}_x : x^\vee \otimes x \rightarrow \onebb$ satisfying again the appropriate two zigzag relations. Equivalently, we could have defined a right dual object for $x \in \Ccal$ to be a left dual object for $x$ in $\Ccal^{rev}$.
A monoidal category is \textit{rigid} if it has both left and right duals. 

Left and right duality can be conveniently expressed in terms of strong monoidal functors $\Ccal^{rev}\rightarrow \Ccal$. To be more precise, the \textit{left dual functor} is defined as
\eq{
\prescript{\vee}{}{(\bullet)}:\Ccal^{rev}&\rightarrow \Ccal\\
x &\mapsto \prescript{\vee}{}{x}\\
\hom_{\Ccal^{rev}}(x,y)\ni f &\mapsto \prescript{\vee}{}{f}\in \hom_\Ccal(\prescript{\vee}{}{x},\prescript{\vee}{}{y})
}
with 
\eq{
\prescript{\vee}{}{f}\coloneqq \left[\prescript{\vee}{}{x}\xrightarrow{\coevrm_y\otimes \idrm_{\prescript{\vee}{}{x}}}\prescript{\vee}{}{y}\otimes y\otimes \prescript{\vee}{}{x}\xrightarrow{\idrm_{\prescript{\vee}{}{y}}\otimes f\otimes \idrm_{\prescript{\vee}{}{x}}} \prescript{\vee}{}{y}\otimes x\otimes \prescript{\vee}{}{x}\xrightarrow{\idrm_{\prescript{\vee}{}{y}}\otimes \evrm_x}\prescript{\vee}{}{y}\right]\quad .
}
Analogously, there is a \textit{right duality functor} 
\eq{
(\bullet)^\vee:\Ccal&\rightarrow \Ccal^{rev}\\
x &\mapsto x^\vee\\
\hom_{\Ccal}(x,y)\ni f &\mapsto f^\vee\in \hom_{\Ccal^{rev}}(x^\vee,y^\vee),
}
where
\eq{
f^\vee\coloneqq\left[y^\vee\xrightarrow{\idrm_{y^\vee}\otimes\widetilde{\coevrm}_x}y^\vee\otimes x\otimes x^\vee\xrightarrow{\idrm_{y^\vee}\otimes f\otimes \idrm_{x^\vee}} y^\vee\otimes y\otimes x^\vee\xrightarrow{\widetilde{\evrm}_y\otimes \idrm_{x^\vee}}x^\vee\right]\quad.
}

It is not hard to show that left and right duality functors are indeed strong monoidal functors. The following coherence result allows us to assume that left and right duality functors are strict and the two functors are  inverse functors:
\begin{lem}\label{duality strictness}\cite[Lemma~5.4]{shimizu2015pivotal}
For any rigid monoidal category $\Ccal$ that is in this lemma not supposed
to be strict, there exists a rigid monoidal category $\Dcal$ 
such that
\begin{enumerate}[label=\roman*)]
\item $\Ccal$ and $\Dcal$ are equivalent as monoidal categories.
\item $\Dcal$ is a strict monoidal category.
\item $\prescript{\vee}{}{(\bullet)}:\Dcal^{rev}\rightarrow \Dcal$ is a strict monoidal functor.
\item $\prescript{\vee}{}{(\bullet)}$ and $(\bullet)^\vee$ are inverse functors.
\end{enumerate}
\end{lem}

\begin{rem}
We could have defined duality functors also with reversed directions, i.e. the left duality functor as functor $\prescript{\vee}{}{(\bullet)}:\Ccal\rightarrow\Ccal^{rev}$ and the right duality functor $(\bullet)^\vee:\Ccal^{rev}\rightarrow \Ccal$. From the previous Lemma, we get $\prescript{\vee}{}{\left((\bullet)^\vee\right)}\simeq \idrm_\Ccal$ and $\left(\prescript{\vee}{}{(\bullet)}\right)^\vee\simeq \idrm_\Ccal$. The double dual functors $\prescript{\vee\vee}{}{(\bullet)}$ 
and $(\bullet)^{\vee\vee}$ are monoidal functors; in general they are 
\textit{not} naturally isomorphic to the identity functor as monoidal functors. 
A pivotal structure amounts to the choice of a monoidal isomorphism; in this 
	paper, we do not require the existence of a pivotal structure.
\end{rem}

\begin{defn} \begin{enumerate}
		\item
	A $\Kbb$-linear category is \textit{finite}, if it is equivalent to the category $A-\Modsf$ of finite-dimensional modules over a finite-dimensional $\Kbb$-algebra $A$.
	\item
	A \textit{finite tensor category} is a finite rigid monoidal category where
	the tensor product is $\Kbb$-bilinear on morphisms.
	\end{enumerate}
\end{defn}

In section \ref{Graphical calculus for Finite Tensor Categories} 
and \ref{framed construction}, we use the term tensor category for a
strict monoidal category \cite{JOYAL199155}.

\begin{rem}
	\begin{enumerate}
		\item For an equivalent intrinsic characterization of finite linear categories, we
		refer to \cite[section~1.8]{egno}. In particular, the morphism spaces
		of a finite category $\Ccal$ are finite-dimensional $\Kbb$-vector spaces and $\Ccal$ has a finite set of isomorphism classes of simple objects.
		\item
		A finite tensor category $\Ccal$ is, in general, neither semi-simple nor pivotal.
	\end{enumerate}

\end{rem}

A linear functor $F:\Ccal\rightarrow \Dcal$ between $\Kbb$-linear categories is not necessarily exact. In case $\Ccal$ and $\Dcal$ are finite tensor categories, it turns out that being left (right) exact is equivalent to admitting a left (right) adjoint.

\begin{theo}\cite[Proposition~1.7]{douglas2019balanced}
A functor $F:\Ccal\rightarrow \Dcal$ between finite linear categories is left (right) exact if and only if it admits a left (right) adjoint. 
\end{theo}

We note several consequences:
by Lemma \ref{duality strictness} the duality functors are inverses and thus adjoints. Hence both functors are exact.  Due to the existence of left and right duals,
the tensor product of a finite tensor category is an exact functor in both elements. 
Finally, given two finite linear
categories $\Dcal, \Ecal$, we denote the category of left exact functors
from $\Dcal$ to $\Ecal$ by $\Lexsf(\Dcal,\Ecal)$.

\subsection{(Co-)Ends in Finite Tensor Categories}

Coends, monads and their module categories will be crucial for relating circle categories obtained from framed string-nets to twisted Drinfeld centers. In
this subsection, we recall necessary definitions and results. Most of the results can be found in \cite[Chapter~VI and IX.6]{mac2013categories}. Throughout this section $\Ccal$ will be a finite tensor category. Some of the results hold in greater generality; we refer to 
\cite[Chapter~IX.6 and IX.7]{mac2013categories}.

Let $\Acal$ be an abelian $\Kbb$-linear category, $H:\Ccal\times \Ccal^{opp}\rightarrow \Acal$ a bilinear bifunctor and $a \in \Acal$ be an object of $\Acal$. A \textit{dinatural transformation from $H$ to $a$} consists of a family of maps $\lbr \psi_c : H(c,c)\rightarrow a\rbr_{c \in \Ccal}$ such that $\psi_d \circ H(f,\idrm_d)=\psi_c \circ H(\idrm_c, f)$ for all $f \in \hom_\Ccal(c,d)$.

\begin{defn} The \textit{coend of $H$} is  an object $\int^{c \in \Ccal} H(c,c)$,
	together with a universal dinatural transformation $\lbr \iota_c : H(c,c)\rightarrow \int^{c \in \Ccal } H(c,c)\rbr$. This means that  for any  
	dinatural transformation $\lbr \psi_c : H(c,c) \rightarrow a\rbr$, there exists a unique morphism \linebreak $\tau\in\hom_\Acal(\int^{c \in \Ccal} H(c,c),a)$ such that the following diagram commutes
\begin{center}
\begin{tikzpicture}
\node at (0,0) (a) {$H(c,d)$};
\node at ($(a)+(6,0)$) (b) {$H(d,d)$};
\node at ($(a)+(0,-4)$) (c) {$H(c,c)$};
\node at ($(c)+(6,0)$) (d) {$\int^{c \in \Ccal} H(c,c)$};
\node at ($(d)+(2,-2)$) (e) {$a$};

\draw[->] (a) to node[above] {$H(f,\idrm_d)$} (b);
\draw[->] (a) to node[left] {$H(\idrm_c,f)$} (c);
\draw[->] (b) to node[right] {$\iota_d$} (d);
\draw[->] (c) to node[below] {$\iota_c$} (d);
\draw[->] (b) to[bend left] node[right] {$\psi_d$} (e);
\draw[->] (c) to[bend right] node[below] {$\psi_c$} (e);
\draw[->] (d) to node[above] {$\tau$} (e);

\end{tikzpicture}
\end{center}

for all $(c,d)\in \Ccal\times\Ccal^{opp}$ and $f: c \to d$.  
\end{defn}

\begin{lem}\label{coend existence}\cite[Corollary~5.1.8]{kerler2001non} If $H:\Ccal\times \Ccal^{opp}\rightarrow \Acal$ is bilinear functor exact in both 
arguments, the coend $\int^{c \in \Ccal} H(c,c)$ exists.
\end{lem}

%
%

\begin{defn}\cite{LyubaRibbon} Let $\Dcal$, $\Ccal$ be finite tensor categories and $\Ecal$ a $\Kbb$-linear category. Assume that the functor 
$H:\Dcal\times\Ccal\times\Ccal^{opp}\rightarrow \Ecal$ is
left exact in both arguments. The \textit{left exact coend of $H$} is an 
object $\oint^{c \in \Ccal} H(\bullet;c,c)$ in the category $\Lexsf(\Dcal,\Ecal)$
of left exact functors, together with a universal dinatural transformation $\big\{ \iota_c:H(\bullet;c,c)\rightarrow \oint^{c \in \Ccal} H(\bullet;c,c)\big\}$ 
consisting of morphisms in $\Lexsf(\Dcal,\Ecal) $.
\end{defn}


\section{Twisted Drinfeld Centers and Monads}
\label{sec: DrinMod}

In this section, we introduce twisted Drinfeld centers of monoidal categories
and review their description as Eilenberg-Moore categories over
monads. String-net constructions do not directly yield Eilenberg-Moore
categories; hence we develop an explicit construction of the Eilenberg-Moore category
of a monad from its Kleisli category.

\subsection{Monadicity of Twisted Drinfeld Centers}\label{twisted Drinfeld}

As before, $\Ccal$ is in this section a finite tensor category. 

The \textit{Drinfeld center} $\Zsf(\Ccal)$ of a monoidal category $\Ccal$ is a categorification of the notion of a center of an algebra. It has as objects pairs $(X,\gamma_{\bullet,x})$, with a natural isomorphism
$
\gamma_{\bullet,x}:\bullet\otimes x\xrightarrow{\simeq}x \otimes \bullet 
$,
called the \textit{half-braiding} such that the identity
$$
\gamma_{c \otimes d,x}=(\gamma_{c,x}\otimes \idrm_{d})\circ (\idrm_c \otimes \gamma_{d,x})
$$
holds for all $c,d \in \Ccal$. The following generalization is well-known:

\begin{defn} Let $F,G:\Ccal\rightarrow \Ccal$ strict $\Kbb$-linear monoidal endofunctors. The \textit{twisted Drinfeld center} $\ZFG$ is the following category:
\begin{itemize}
\item \textit{Objects} are pairs  $(x,\gamma_{\bullet,x})$, where 
\eq{
\gamma_{\bullet,x}:F(\bullet)\otimes x \xrightarrow{\simeq} x \otimes G(\bullet)
}
is a natural isomorphism satisfying
\eq{\label{Drinfeld Hexagon}
\gamma_{c\otimes d,x}=(\gamma_{c,x}\otimes \idrm_{G(d)})\circ (\idrm_{F(c)}\otimes \gamma_{d,x})
}
for all $c,d\in\Ccal$.
\item A \textit{morphism} $f: (x,\gamma_{\bullet,x}) \to (y,\gamma_{\bullet,y})$ is a morphism $f \in \hom_\Ccal(x,y)$ such that
\eq{
\left[F(c)\otimes x \xrightarrow{\gamma_{c,x}}x\otimes G(c)\xrightarrow{f\otimes \idrm}y\otimes G(c)\right]=\left[F(c)\otimes x\xrightarrow{\idrm\otimes f}F(c)\otimes y\xrightarrow{\gamma_{c,y}}y\otimes G(c)\right]\; . 
} 
\end{itemize} 
\end{defn}

The monoidal functors we will be interested in are powers of the double duals.
Specifically, we consider the following cases
\eq{\label{Zn definition}
	\Zsf_n(\Ccal)\coloneqq \begin{cases} {}_{\left({}^{\vee\vee}{}{(\bullet)}\right)^{n-1}}{\Zsf}_{\idrm_\Ccal}(\Ccal),\; &n\in \Zbb_{>0}, \\
		{}_{\left(\bullet\right)^{\vee\vee}}{\Zsf}_{\idrm_\Ccal}(\Ccal), &n=0, \\
		{}_{(\bullet)^{\vee\vee}}{\Zsf}_{\left({}^{\vee\vee}{}{(\bullet)}\right)^{-n}}, &n\in \Zbb_{<0},
	\end{cases}
}
which include for $n=1$ the usual Drinfeld center $\Zsf(\Ccal)$.
The category ${}_{(\bullet)^{\vee\vee}}{\Zsf}_{\idrm_\Ccal}(\Ccal)$ 
obtained for $n=0$ is known as the \textit{trace} of $\Ccal$, see e.g. \cite[Definition 3.1.4]{douglas2020dualizable}. 

These categories can be described in terms of monads on $\Ccal$. A \textit{monad} on a category $\Ccal$ is a triple $(T,\mu,\eta)$ consisting of an endofunctor $T: \Ccal \rightarrow \Ccal$ and natural transformations $\mu: T^2 \Rightarrow T$, $\eta : \idrm_\Ccal \Rightarrow T$ such that the diagrams
\begin{center}
\begin{tikzcd} 
T^3(c)\ar[r,"T(\mu_c)"] \ar[d,"\mu_{T(c)}"']& T^2(c)\ar[d,"\mu_c"]\\
T^2(c)\ar[r,"\mu_c"'] & Tc
\end{tikzcd}\hspace*{2cm}
\begin{tikzcd}
Tc\ar[r,"\eta_{T(c)}"] \ar[dr,"\idrm"'] & T^2(c)\ar[d,"\mu_c"] & Tc\ar[l,"T(\eta_c)"']\ar[ld,"\idrm"]\\
& Tc & 
\end{tikzcd}
\end{center}
\noindent
commute for all $c \in \Ccal$. A \textit{module} for the monad $(T,\mu,\eta)$ is a pair $(d,\rho)$, consisting of an object $d \in \Ccal$ and a morphism $\rho : Td \rightarrow d$ such that the diagrams

\begin{center}
\begin{tikzcd}
T^2(d)\ar[r,"\mu_d"]\ar[d,"T(\rho)"'] & Td\ar[d,"\rho"]\\
Td\ar[r,"\rho"'] & d
\end{tikzcd}\hspace*{2cm}
\begin{tikzcd}
d \ar[r,"\eta_d"]\ar[dr,"\idrm"'] & Td\ar[d,"\rho"]\\
 & d
 \end{tikzcd}
\end{center}

\noindent 
commute. A \textit{morphism between two $T$-modules $(d_1,\rho)$, $(d_2,\lambda)$} is a morphism $f \in \hom_\Ccal(d_1,d_2)$ such that the diagram
\begin{center}
\begin{tikzcd}
Td_1\ar[d,"T(f)"']\ar[r,"\rho"] & d_1 \ar[d,"f"]\\
Td_2\ar[r,"\lambda"'] & d_2
\end{tikzcd}
\end{center}
commutes.

\noindent
We denote the category of $T$-modules or \textit{Eilenberg-Moore category}  by
$T-\Modsf$ or $\Ccal^T$. It comes with a forgetful functor $U^T$ to $\Ccal$.

Given two exact $\Kbb$-linear strict monoidal endofunctors $F,G$ of a finite 
tensor category $\Ccal$, the functor
\eq{
Q:\Ccal \times \Ccal^{opp} &\rightarrow \Funsf(\Ccal,\Ccal)\\
 (c,d)&\mapsto F(c)\otimes\bullet\otimes G(\prescript{\vee}{}{d})
 }
is exact in both arguments.
 Thus, by Lemma \ref{coend existence}, the coend
\eq{
\TFG(\bullet)\coloneqq\int^{c\in \Ccal}F(c)\otimes\bullet\otimes G(\prescript{\vee}{}{c})\in \Funsf(\Ccal,\Ccal)
}
exists. It is a known fact (cf. \cite[Section~3.3]{shimizu2017ribbon}) that $\TFG(\bullet)$ is a monad in $\Ccal$ with multiplication 
induced by the dinatural family
\eq{
\Big\{ \big[ F(d)\otimes F(c)\otimes \bullet \otimes G(\prescript{\vee}{}{c})\otimes G(\prescript{\vee}{}{d}) & = F(d\otimes c)\otimes \bullet \otimes G(\prescript{\vee}{}{(d\otimes c)}) \\
&  \xrightarrow{\iota_{d\otimes c}} \int^{a\in \Ccal}F(a)\otimes\bullet\otimes G(\prescript{\vee}{}{a}) \big] \Big\}_{c,d\in\Ccal}
}
where $\iota$ is the dinatural family of the coend $\TFG$. Associativity of the multiplication follows from the Fubini theorem 
\cite[Chapter~IX.7]{mac2013categories}
for iterated coends. The following proposition relates twisted Drinfeld
centers to Eilenberg-Moore categories of this monad, which we call the \textit{twisted
central monad}:

\begin{prop}\cite[Lemma~3.8]{shimizu2017ribbon}\label{monadicity}
There is an isomorphism of $\Vectsf_\Kbb$-enriched categories \eq{
\TFG-\Modsf\simeq \ZFG
}
commuting with forgetful functors
\begin{center}
\begin{tikzcd}
\TFG-\Modsf\ar[rr,"\simeq"] \ar[rdd,"U^{T}"']& & \ZFG \ar[ddl]\\
& & \\
& \Ccal & .
\end{tikzcd}
\end{center}
\end{prop}

We denote by $T_n$ the monad on $\Ccal$ describing
the Drinfeld center $\Zsf_n(\Ccal)$ twisted by a power of the bidual,
cf.\ (\ref{Zn definition}).
Proposition \ref{monadicity} is  a statement about $\Vectsf_\Kbb$-enriched categories. However, the following corollary is immediate from the proposition and \cite[Lemma~2.7]{shimizu2016unimodular}, as $\TFG$-is a right exact functor.

\begin{Cor}\label{Z abelian}
$\ZFG$ is a finite $\Kbb$-linear category. 
\end{Cor}

\begin{lem}
$\TFG^{}: \Ccal \rightarrow \Ccal$ is an exact functor.
\end{lem}

\begin{proof}
Recall that $F$, $G$ are assumed to be exact functors and exact functors commute with (co-)limits. By e.g. \cite[Section~1.2]{loregian2021} a coend is a colimit, thus 
we have
\eq{
\TFG(\bullet)=\left(F\otimes\idrm_\Ccal\otimes G\right) \circ \left(\int^{c\in\Ccal} c\otimes (\bullet)\otimes \prescript{\vee}{}{c}\right)=\left(F\otimes\idrm_\Ccal\otimes G\right) \circ T(\bullet),
}
where $T=\int^{c\in\Ccal} c\otimes (\bullet)\otimes \prescript{\vee}{}{c}$. Hence, $\TFG$ is exact if and only if $T$ is exact. As $T=U\circ T^f$, with $T^f$ the left adjoint of the exact forgetful functor $U:T-\Modsf\rightarrow \Ccal$, this holds, if and only if $T^f$ is exact. Exactness of $T^f$ is shown in \cite[Corollary~4.9]{shimizu2016unimodular}.
\end{proof}

In section \ref{Cylinder Kleisli}, we need the following result. 

\begin{lem}\label{Kleisli left exact coend}
Let $\Ccal$ be a finite tensor category and $F,G\in \Funsf(\Ccal,\Ccal)$ exact strict monoidal endofunctors. Let
\eq{
Q:\Ccal \times \Ccal^{opp} &\rightarrow \Lexsf(\Ccal \times \Ccal^{opp},\Vectsf_\Kbb)\\
(c,d)&\mapsto  \hom_\Ccal(\,(\bullet)\,,F(c)\otimes(\bullet)\otimes G(\prescript{\vee}{}{d}))\; .
}
Then the left exact coend $\oint^{c\in\Ccal}Q(c,c)$ exists and there is an isomorphism
\eq{
\oint^{c\in\Ccal}Q(c,c)(\bullet,\bullet)\simeq \hom_\Ccal(\, (\bullet)\,,\TFG(\bullet))\;  .
}
\end{lem}

\begin{proof}
Since $\TFG$ is an exact functor, $\hom_\Ccal(\,(\bullet)\,,\TFG(\bullet)):\Ccal \times \Ccal^{opp} \rightarrow \Vectsf_\Kbb$ is left exact. Therefore it suffices to show that $\hom_\Ccal(\,(\bullet)\, ,\TFG(\bullet))$ has the universal property of the left exact coend. This can be proven along the lines of \cite[Proposition~9]{fuchs2016coends}. Adapting the proof given there to the current situation is not hard and is left as an exercise to the reader. 
\end{proof}

\subsection{Kleisli Categories and Representable Functors}
\label{Kleisli and flat functors}

The string-net construction
will not directly give the twisted center $\Zsf_n(\Ccal)$. Hence
we recall that given any monad $(T,\mu,\eta)$,
there are several adjunctions giving rise to the same monad. In this subsection,
we review 
this theory for a general monad $T$ which is not necessarily a twisted central
monad; for a textbook account, we refer to \cite[Chapter 5]{RIeh2}.

\begin{itemize}
	\item As discussed in subsection \ref{twisted Drinfeld}, the category of $T$-modules $\Ccal^T$
	 has as objects pairs $(c,\rho)$ with
	$c\in\Ccal$ and $\rho: Tc \to c$ a morphism in $\Ccal$. The forgetful
	functor $U^T:\Ccal^T\to\Ccal$ assigns to a $T$-module $(c,\rho)$ 
	the underlying object $c\in\Ccal$.
	Its left adjoint $I^T: \Ccal\to\Ccal^T$ assigns to $c\in\Ccal$ the free module
	$Tc$ with action $\mu_c:T^2(c)\to Tc$. The monad $U^T\circ I^T$  induced
	on $\Ccal$  by
	the adjunction $I^T\dashv U^T$ is again $T$. 
	
	\item The \textit{Kleisli category $\Ccal_T$} has as objects the objects
	of $\Ccal$; whenever an object $c\in\Ccal$ is seen as an object of the
	Kleisli category $\Ccal_T$, it will be denoted by $\ov c$. 
	The $\hom$-spaces of the Kleisli category are
	$\hom_{\Ccal_T}(\ov{c},\ov{d})\coloneqq \hom_{\Ccal}(c,Td)$, for all
	$c,d\in \Ccal$. A morphism
	in $\Ccal_T$ from $\ov c$ to $\ov d$ will be denoted by
	$\ov c\rightsquigarrow \ov d$.
	The composition of morphisms in the Kleisli category $\Ccal_T$
	is 
	\eq{
		g\circ_{\Ccal_T} f\coloneqq \mu_{c_3}\circ _\Ccal T(g)\circ_\Ccal f
	}
	for $g: \ov c_2\rightsquigarrow \ov c_3$ and $f: \ov c_1 \rightsquigarrow \ov c_2$.
	The identity morphism $\ov c \rightsquigarrow \ov c$ in $\Ccal_T$ is,
	 as a morphism in $\Ccal$, the component $\eta_c: c\to Tc$ of the unit of $T$.
	
	Define a functor $I_T:\Ccal\to\Ccal_T$ which is the identity on objects
	and sends a morphism $c_1\stackrel f\to c_2$ in $\Ccal$ to the 
	morphism $\overline c_1
	\rightsquigarrow \ov c_2$ given by the morphism
	$$ I_T (f):\quad c_1 \stackrel f\to  c_2 \stackrel{\eta_{c_2}}\to Tc_2 \,\, $$
	 in $\Ccal$. Define also a functor $U_T: \Ccal_T\to\Ccal$ sending
	$\ov c\in\Ccal_T$ to $Tc\in\Ccal$ and a morphism $\ov h:\,\,\overline c
	\rightsquigarrow \ov d$ represented by the morphism $h: c\to Td$ in $\Ccal$ to
	$$ U_T(\ov h): \quad Tc\stackrel{Th}\to T^2(d)\stackrel{\mu_{d}}\to Td \,\,.$$
	By \cite[Lemma 5.2.11]{RIeh2}, this gives a pair of
	adjoint functors, $I_T\dashv U_T$, and
	that the adjunction realizes again the monad $T$ on $\Ccal$, i.e.\ $U_T\circ I_T=T$.	
	
	\item
	It is also known \cite[Proposition 5.2.12]{RIeh2} that the Kleisli
	category $\Ccal_T$ is initial and that the Eilenberg-Moore category
	$\Ccal^T$ is final in the category of adjunctions realizing the monad $T$ on $\Ccal$.
	Put differently, for any adjunction $\Dcal\stackrel{U}\to\Ccal$ and
	$\Ccal\stackrel{I}\to \Dcal$ with $I\dashv U$ and $U\circ I=T$, 
	there are unique comparison functors $K_\Dcal:\Ccal_T\to \Dcal$ and
	$K^\Dcal: \Dcal\to\Ccal^T$	such that the diagram
\begin{center}
	\begin{tikzcd}
	\Ccal_T \ar[rr,"K_\Dcal"]
	\ar[ddrr, bend left=10, "U_T"]
	& &
	\Dcal \ar[rr,"K^\Dcal"] \ar[dd, bend left=10, "U"]
	& &\Ccal^T \ar[lldd, bend left=10, "U^T"]
	\\ \phantom{bla}&&\phantom{bla}&& \\
	&&\Ccal\ar[uu, bend left=10, "I"]\ar[uurr, bend left=10, "I^T"]
	\ar[uull, bend left=10, "I_T"]&& 
   \end{tikzcd}
\end{center}		
	commutes. 
	\item
	An adjunction $I\dashv U$ that induces the monad $T=U\circ I$ 
	on $\Ccal$ is called \textit{monadic}, if 
	the comparison functor $K^\Dcal$ to the Eilenberg-Moore category $\Ccal^T$
	is an equivalence of categories.
\end{itemize}

From the string-net construction, we will recover in Theorem \ref{main theo}
the Kleisli categories of the twisted central monads as circle categories.
If $\Ccal$ is semi-simple, the twisted Drinfeld center can then be
recovered as a Karoubification \cite{kirillov2011string} or
as presheaves \cite{hoek}. For non-semi-simple categories, this does
not suffice. It is instructive to understand how
to explicitly recover  the Eilenberg-Moore category from the
Kleisli category.

Recall that all categories are linear and all functors are linear functors.
Denote by $\Dcal:=\mathrm{PSh}_{I_T}(\Ccal_T)$ the category of functors
$F:=\Ccal_T^{opp}\to\Vectsf_\Kbb$ such that the pullback by $I_T$
$$F\circ I_T^{opp}: \,\, \Ccal^{opp}\xrightarrow{I_T^{opp}}\Ccal_T^{opp}
\xrightarrow{F}\Vectsf_\Kbb $$
is representable by some object $c_F\in\Ccal$.  We then say that $F\in\Dcal$ is
an $I_T$-representable presheaf on the Kleisli category $\Ccal_T$.
In this way, we obtain a functor $U: \Dcal\to \Ccal$ sending the 
presheaf $F$ to the $I_T$-representing object $c_F\in\Ccal$.

We construct its left adjoint:
For $c\in\Ccal$, consider the functor $\hom_{\Ccal_T}(\bullet,I_Tc): \Ccal_T^{opp}\to \Vectsf _\Kbb$. The pullback of this functor along $I_T$ is representable, as follows from
the equivalences
$$ \hom_{\Ccal_T}(I_T\bullet,I_Tc) \cong \hom_\Ccal(\bullet,U_TI_Tc)
\cong \hom_\Ccal(\bullet,Tc) \,\,. $$
Note that the $I_T$-representing object of $\hom_{\Ccal_T}(\bullet,I_Tc)$ is
$Tc\in\Ccal$. We thus obtain a functor
$$ \begin{array}{rll}
I: \,\, \Ccal &\to& \Dcal \\
c&\mapsto&  \hom_{\Ccal_T}(\bullet,I_Tc)\,\,.
\end{array}$$
We have already seen that $U\circ I=T$.  It remains to see that the functors $I$ and
$U$ are adjoint,
$$\hom_\Dcal(Ic,F)\cong \hom_\Ccal(c,U (F))  \,\,,$$
where $F\in\Dcal$ is assumed to be $I_T$-representable by $c_F\in\Ccal$. Hence the
right-hand side is naturally isomorphic to $\hom_\Ccal(c,c_F)$.
For the left-hand side, we compute
$$\begin{array}{rll}
\hom_\Dcal(Ic,F)&=&\Natsf(Ic,F)= \Natsf(\hom_{\Ccal_T}(\bullet,I_Tc),F)
\cong F(I_Tc) \\
&=& \hom_\Ccal(c,c_F)
\end{array}$$
where in the first line we used the Yoneda lemma and in the second line that
$F\circ I_T$ is represented by $c_F\in\Ccal$.

We are now ready for the main result of this subsection:

\begin{prop}\label{T-mod is pullback} 
The adjunction $I\dashv U$ with
$U: \mathrm{PSh}_{I_T}(\Ccal_T)\to \Ccal$  and
$I: \Ccal\to\mathrm{PSh}_{I_T}(\Ccal_T)$ is monadic. As a consequence,
the comparison functor $K: \mathrm{PSh}_{I_T}(\Ccal_T)\to \Ccal^T$
is an equivalence of categories and the Eilenberg-Moore category
can be identified with the category of $I_T$-representable presheaves on
the Kleisli category $\Ccal_T$.
\end{prop}

In \cite{STREET1972149} Proposition \ref{T-mod is pullback} is proven 
in a more general setting, using bicategorical methods. The statement of
Proposition \ref{T-mod is pullback} appears as a comment
in \cite[Exercise 5.2.vii]{RIeh2}.
For the convenience of the reader, we give an explicit proof, using
the monadicity theorem \cite[Theorem 5.5.1]{RIeh2}. 

\begin{proof}
Recall the shorthand	$\Dcal:=\mathrm{PSh}_{I_T}(\Ccal_T)$. 
We have to show that $U:\Dcal\to\Ccal$ creates coequalizers of
$U$-split pairs. Thus, consider for two $I_T$-representable functors
$F_1,F_2\in\Dcal$ a parallel pair
\begin{center}
	\begin{tikzcd}
		F_1 \ar[r, "\nu_1", bend left]  \ar[r,swap, "\nu_2", bend right] 
		& F_2 
	\end{tikzcd}
\end{center}	
of natural transformations 
and assume that for $c_i:=U(F_i)\in\Ccal$ and $n_i:=U(\nu_i)$ 
for $i=1,2$ there is a  split equalizer in $\Ccal$ for the parallel pair $n_1,n_2$:
\begin{equation}
	\begin{tikzcd}
	c_1 \ar[r, "n_1", bend left]  \ar[r, bend right, swap, "n_2"]
	& c_2 \ar[r,"h"] \ar[l] & c_3\,\,. \ar[l,bend left]
\end{tikzcd}\label{split coeq}
\end{equation}	

We have to find a coequalizer $\mathrm{coeq}(\nu_1,\nu_2): F_2\to F_3$ 
in $\Dcal$
such that $U(F_3)=c_3$ and the coequalizer is mapped by $U$ to $h$.
The functors are linear and natural transformations
are vector spaces; hence we can consider the natural
transformation $\nu:=\nu_1-\nu_2: F_1\to F_2$ and determine
its cokernel in $\Dcal$. We also introduce the notation $n:=n_1-n_2: c_1\to c_2$.

We start by defining a functor $F_3: \Ccal_T^{opp}\to\Vectsf_\Kbb$
on an object $\ov \gamma\in\Ccal_T^{opp}$ as the cokernel of the components
of $\nu$
in the category of vector spaces, so that we have for each $\ov \gamma\in\Ccal_T^{opp}$ 
an exact sequence
\begin{center}
	\begin{tikzcd}
F_1(\ov \gamma) \ar[r, "\nu_{\ov\gamma}"]& F_2(\ov \gamma)
\ar[r, "q_{\ov \gamma}"] &F_3(\ov \gamma)\ar[r]&  0
	\end{tikzcd}
\end{center}	
in the category of vector spaces.
To define the functor $F_3$ on a morphism $\ov\gamma_1\stackrel{f}\to 
\ov\gamma_2$ in $\Ccal_T^{opp}$,
consider the diagram
\begin{center}
	\begin{tikzcd}
		F_1(\ov\gamma_1) \ar[r,"\nu_{\ov\gamma_1}"] \ar[d,"F_1(f)"]&F_2(\ov\gamma_1) \ar[r,"q_{\ov\gamma_1}"]
		\ar[d,"F_2(f)"]& F_3(\ov\gamma_1) \ar[r]\ar[d,dashrightarrow]&0\\
		F_1(\ov\gamma_2) \ar[r,"\nu_{\ov\gamma_2}"] &F_2(\ov\gamma_2) \ar[r,"q_{\ov\gamma_2}"]
		& F_3(\ov\gamma_2) \ar[r]&0
	\end{tikzcd}
\end{center}	
which has, by definition, exact rows.
The left square commutes because of the naturality of $\nu$. A standard diagram
chase shows that there exists a unique linear map for the dashed arrow which
we denote by $F_3(f)$. This completes $F_3$ to a functor $\Ccal_T^{opp}\to\Vectsf_\Kbb$
and shows that the components $(q_{\ov\gamma})_{\ov\gamma\in\Ccal^T}$
assemble into a natural transformation $q: F_2\to F_3$.

We have to show that the functor $F_3$ is $I_T$-representable and indeed 
represented by the object $c_3$ appearing in the split coequalizer
(\ref{split coeq}). To this end, consider the two pullbacks 
\begin{center}
	\begin{tikzcd}
\tilde F_i:=F_i\circ I_T^{opp}: \,\,\, \Ccal^{opp} \ar[r]& \Ccal_T^{opp}\ar[r]& \Vectsf_\Kbb
	\end{tikzcd}
\end{center}	
which come with isomorphisms
$$ \phi_i: \quad \tilde F_i \stackrel\sim\to \hom_\Ccal(\bullet,c_i) $$
of functors for $i=1,2$. For each $\gamma\in\Ccal$, we get a commuting diagram
\begin{equation}
	\begin{tikzcd}
\tilde F_1(\gamma)		\ar[r,"\nu_{I_T\gamma}"]\ar[d, "(\phi_1)_\gamma"]& 
\tilde F_2(\gamma) \ar[r,"q_{I_T\gamma}"]\ar[d, "(\phi_2)_\gamma"]&
\tilde F_3(\gamma) \ar[r] \ar[d, dashrightarrow]&0 \\
\hom_\Ccal(\gamma,c_1)\ar[r,"n_*"] &
\hom_\Ccal(\gamma,c_2)\ar[r,"h_*"]&
\hom_\Ccal(\gamma,c_3)\ar[r]&0. 
\end{tikzcd}\label{comm diagram}
\end{equation}	
The upper row is exact by construction. The lower row is exact, since $c_3$ was
part of a split coequalizer in $\Ccal$ and split coequalizers are preserved by
all functors. Again, a diagram chase implies the existence of
a morphism $(\phi_3)_\gamma: \tilde F_3(\gamma)\to \hom_\Ccal(\gamma,c_3)$
for the dashed arrow which by the nine lemma is an isomorphism.

To show the naturality of the morphisms $(\nu_3)_\gamma$, we take
a morphism $\gamma_1\stackrel{f}\to\gamma_2$ in $\Ccal^{opp}$ and
consider the diagram which consists of two adjacent cubes and four more
arrows:
\begin{center}
\begin{tikzcd}[row sep=scriptsize, column sep=tiny,nodes={inner sep=0pt},arrows={shorten =2pt}]
\tilde F_1(\gamma_1) \ar[rr,"\nu"]\ar[dr] \ar[dd,"\phi_1"]&& 
\tilde F_2(\gamma_1)\ar[rr,"q"]\ar[dr] \ar[dd]&&
\tilde F_3 (\gamma_1)\ar[rr] \ar[dr]\ar[dd]&& 0& \\
&\tilde F_1(\gamma_2) \ar[rr, "\nu" near start,crossing over]
&& 
\tilde F_2(\gamma_2)\ar[rr, "q" near start, crossing over] &&
\tilde F_3 (\gamma_3)\ar[rr] && 0 \\
\hom(\gamma_1,c_1)\ar[rr,"n_*" near start]\ar[dr] && 
\hom(\gamma_1,c_2)\ar[rr, "h_*" near start]\ar[dr]&&
\hom(\gamma_1,c_3)\ar[rr]\ar[dr]&& 0& \\
&\hom(\gamma_2,c_1)\ar[rr, swap, "n_*"] 
\ar[from=uu, "\phi_1" near start,  crossing over]&& 
\hom(\gamma_2,c_2)\ar[rr,swap, "h_*"]
\ar[from=uu, "\phi_2" near start,  crossing over]&&
\hom(\gamma_2,c_3)\ar[rr]\ar[from=uu, "\phi_3" near start,  crossing over]&& 0.
\end{tikzcd}
\end{center}	
To keep the diagram tidy, we do not provide all labels of the arrows and explain them
here: diagonal arrows are labeled by applying the appropriate functor to
$f:\gamma_1\to\gamma_2$. Vertical arrows are isomorphisms labeled
by $\phi_i$. The front and rear squares of the two cubes are just instances
of the commuting diagram (\ref{comm diagram}) and thus commute. 
The squares on the
top commute because $\nu$ and $q$ are natural; similarly, the squares on the bottom
commute because $n_*$ and $h_*$ are natural. The left and middle diagonal walls
commute because $\phi_1$ and $\phi_2$ are natural. A diagram chase now
yields that the rightmost wall commutes as well, which is the naturality of
$\phi_3$.

\end{proof}


\section{Progressive Graphical Calculus for  Tensor Categories}\label{Graphical calculus for Finite Tensor Categories}

It is standard to introduce a graphical calculus for  computations in (strict)  tensor categories. Following \cite{JOYAL199155}, morphisms in a (strict)  tensor category $\Ccal$ can be represented by  so-called \textit{progressive graphs}
on a standard rectangle in the $x-y$-plane.

A \textit{graph} is a $1$-dimensional, finite CW-complex $\Gamma$ with a finite, closed subset $\Gamma_0\subset \Gamma$ such that $\Gamma-\Gamma_0$ is a $1$-dimensional smooth manifold without boundary. Elements of $\Gamma_0$ are called \textit{nodes} of the graph. A node $b$ is a \textit{boundary node}, if for any connected open neighborhood $b\in U\subset \Gamma$, $U-\lbr b\rbr $ is still connected. The collection of boundary nodes is called the \textit{boundary of $\Gamma$} and is denoted by $\p \Gamma$. An \textit{edge} is a connected component $e\subset \Gamma- \Gamma_0$ homeomorphic to the interval $(0,1)$. By adjoining its endpoints to $e$, we get a closed edge $\hat{e}$. An \textit{oriented edge} is an edge with an orientation. For an oriented edge $\hat{e}$ we admit only homeomorphism $\hat{e}\simeq [0,1]$ preserving orientations.  The endpoints of $\hat{e}$ then are linearly ordered: The preimage of $0$ in $\hat{e}$, denoted by $\hat{e}(0)$, is the source and the preimage $\hat{e}(1)$ of $1$ is the target. A graph where every edge is endowed with an orientation is called an oriented graph. For an oriented graph, an edge $e$, adjacent to a node $v$, is \textit{incoming at $v$}, if $v$ is the target of $e$ and \textit{outgoing}, if $v$ is the source of $e$. This gives two, not necessarily disjoint, subsets $\mathrm{in}(v)$ and $\mathrm{out}(v)$ of incoming and outgoing edges at $v$. An oriented graph $\Gamma$ is \textit{polarized}, if for any $v\in \Gamma$, $\mathrm{in}(v)$ and $\mathrm{out}(v)$ are linearly ordered sets.  

\begin{defn} Let $(\Gamma,\Gamma_0,\p\Gamma)$ be a polarized graph and $(\Ccal,\otimes, \onebb)$ a monoidal category. A \textit{$\Ccal$-coloring} of $\Gamma$ comprises two functions
\eq{
\varphi_0:\Gamma-\Gamma_0\rightarrow \mathrm{ob}(\Ccal),\qquad \varphi_1:\Gamma_0-\p\Gamma\rightarrow \mathrm{mor}(\Ccal)
}
associating to any oriented edge of $\Gamma$ an object of $\Ccal$ and to any inner node $v\in \Gamma_0-\p\Gamma$ a morphism in $\Ccal$, with 
\eq{
\varphi_1(v):\varphi_0(e_1)\otimes \cdots \otimes \varphi_0(e_n)\rightarrow \varphi_0(f_1)\otimes \cdots\otimes \varphi_0(f_m),
}
where $e_1<\cdots < e_n$ and $f_1<\cdots < f_m$ are the ordered elements of $\mathrm{in}(v)$ and $\mathrm{out}(v)$, respectively.
\end{defn}

\begin{defn} A \textit{planar} graph is a graph $(\Gamma, \Gamma_0,\p\Gamma)$ together with a smooth embedding $\iota:\Gamma\rightarrow \Rbb^2$.
\end{defn}

For a planar graph, we will not distinguish in our notation
between the abstract graph $\Gamma$ and its embedding $\iota(\Gamma)$.
Note that a graph has infinitely many realizations as a planar graph, by choosing different embeddings.

\begin{defn} Let $a,b\in \Rbb$ with $a<b$. A \textit{progressive graph} in $\Rbb\times[a,b]$ is a planar graph $\Gamma\subset \Rbb\times [a,b]$ such that
\begin{enumerate}[label=\roman*)]
\item All outer nodes are either on $\Rbb\times \lbr a\rbr $ or on $\Rbb\times \lbr b\rbr$, i.e.
\eq{
\p \Gamma=\Gamma\cap(\Rbb\times\lbr a, b\rbr) \,\,.
}
\item The restriction of  the projection to the second component
\eq{
\prrm_2:\Rbb\times [a,b]\rightarrow [a,b]
}
to any connected component of $\Gamma-\Gamma_0$ is an injective map.
\end{enumerate}
\end{defn}

\begin{rem}
Using the injective projection to the second component, every progressive graph is oriented. In addition, it is also polarized. For any $v\in \Gamma_0$, we can pick $u\in [a,\prrm_2(v))$ such that any element of $\mathrm{in}(v)$ intersects $\Rbb\times \lbr u\rbr $. Since the graph is progressive, the intersection points are unique. The intersection points of $\mathrm{in}(v)$ with $\Rbb\times \lbr u\rbr $ are linearly ordered by the orientation of $\Rbb$ and induce a linear order on $\mathrm{in}(v)$. Similar, one defines a linear order on $\mathrm{out}(v)$ using the intersection with $\Rbb\times \lbr w\rbr $, for $w\in (\prrm_2(v),b]$.
\end{rem}

\begin{rem}
 A progressive graph cannot have cups, caps or circles, since the restriction of $\prrm_2$ to these would be non-injective. This mirrors the fact that in a
	general  non-pivotal category left and right duals for an object are not isomorphic 
and  there are no categorical traces.
Thus we should not represent (co-)evaluation morphisms simply by oriented cups and caps, but use explicitly labeled coupons. 
In addition, in the absence of a categorical trace, we cannot make sense of a circle-shaped diagram.
\end{rem}

Since a progressive graph $\Gamma$ is always polarized, we have a notion
of a $\Ccal$-coloring for it, where $\Ccal$ is a monoidal category. 
Given a $\Ccal$-coloring $\varphi\coloneqq(\varphi_0,\varphi_1)$ of $\Gamma$, we associate to every boundary node $v\in \p \Gamma$ the object in $\Ccal$ of its adjacent edge. The \textit{domain} $\domrm(\Gamma,\varphi)$ of $\Gamma$
is the linearly ordered set of objects assigned to the boundary node in $\Rbb\times\lbr a\rbr$. Its \textit{codomain} $\codomrm(\Gamma,\varphi)$ is the linearly ordered set of objects assigned to the boundary nodes in $\Rbb\times \lbr b\rbr $. 

To the pair $(\Gamma,\varphi)$ of a progressive graph $\Gamma$ with $\Ccal$-coloring $\varphi$ and $\domrm(\Gamma,\varphi)=(X_1,\cdots, X_n)$ and $\codomrm(\Gamma, \varphi) = (Y_1,\cdots, Y_m)$, we can associate a morphism in $\Ccal$
\eq{
f_\Gamma:X_1 \otimes \cdots \otimes X_n\rightarrow Y_1 \otimes \cdots \otimes Y_m \ .
}
The full technical details of this construction can be found in \cite{JOYAL199155}. We will discuss it for an example, the general procedure will then be clear. \newpage

Let $(\Gamma,\Gamma_0,\p\Gamma)$ be the following $\Ccal$-colored progressive graph:
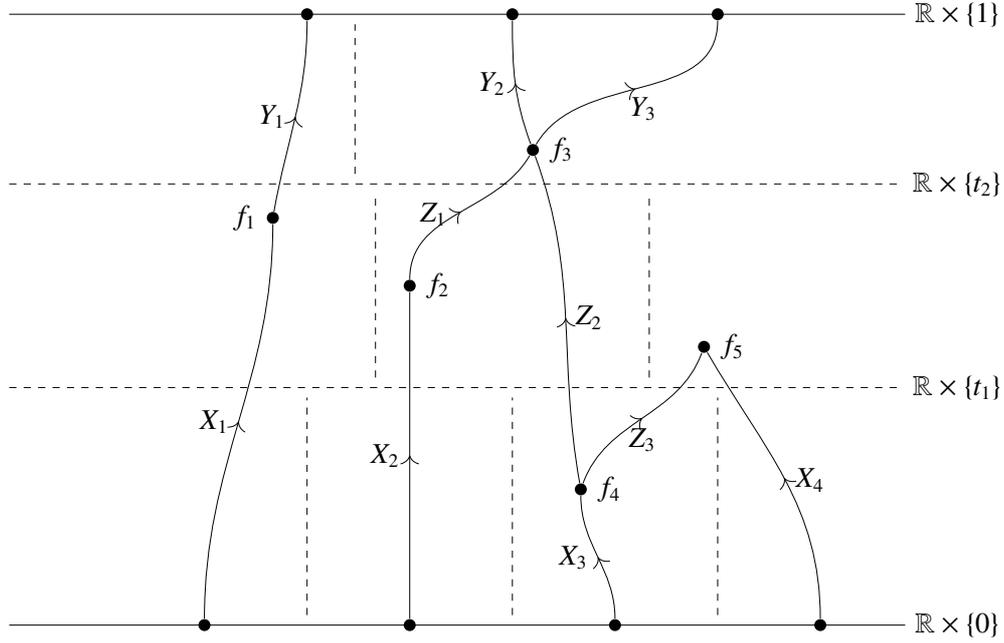
\begin{figure}[H]
    \makebox[\textwidth][c]{
    \resizebox{0.9\width}{!}{
    \begin{tikzpicture}
        
        \node[dot,label=left:$f_1$] at (-1,0) (a) {};
        \node[dot,label=right:$f_2$] at ($(a)+(2,-1)$) (b) {};
        \node[dot,label=right:$f_3$] at ($(b)+(1.8,2)$) (c) {};
        \node[dot, label=right:$f_4$] at ($(b)+(2.5,-3)$) (d) {};
        \node[dot, label=right:$f_5$] at ($(d)+(1.8,2.1)$) (e) {};
        \node[dot] at (-2,-6) (b1) {};
        \node[dot] at ($(b1)+(3,0)$) (b2) {};
        \node[dot] at ($(b2)+(3,0)$) (b3) {};
        \node[dot] at ($(b3)+(3,0)$) (b4) {};
        
        \node[dot] at (-0.5,3) (t1) {};
        \node[dot] at ($(t1)+(3,0)$) (t2) {};
        \node[dot] at ($(t2)+(3,0)$) (t3) {};
        
        \node at (-5,-6) (bl) {};
        \node at ($(bl)+(14,0)$) (br) {$\Rbb\times \lbr 0\rbr $};
        \node at (-5,3) (tl) {};
        \node at ($(tl)+(14,0)$) (tr) {$\Rbb\times \lbr 1\rbr $};
        \node at (-5,-2.5) (m1l) {};
        \node at ($(m1l)+(14,0)$) (m1r) {$\Rbb\times \lbr t_1\rbr $};
        \node at (-5,0.5) (m2l) {};
        \node at ($(m2l)+(14,0)$) (m2r) {$\Rbb\times \lbr t_2\rbr $};
        
        \node at ($(b1)+(1.5,0)$) (s1b) {};
        \node at ($(s1b)+(0,3.5)$) (s1t) {};
        \node at ($(b2)+(1.5,0)$) (s2b) {};
        \node at ($(s2b)+(0,3.5)$) (s2t) {};
        \node at ($(b3)+(1.5,0)$) (s3b) {};
        \node at ($(s3b)+(0,3.5)$) (s3t) {};
        \node at ($(s1t)+(1,0)$) (s4b) {};
        \node at ($(s4b)+(0,3)$) (s4t) {};
        \node at ($(s2t)+(2,0)$) (s5b) {};
        \node at ($(s5b)+(0,3)$) (s5t) {};
        \node at ($(s4t)+(-0.3,0)$) (s6b) {};
        \node at ($(s6b)+(0,2.5)$) (s6t) {};

		\draw[mid] (b1) to[out=90,in=270] node[left] {$X_1$} (a) ;
		\draw[mid] (a) to[out=80,in=270] node[left] {$Y_1$}(t1);
		\draw[mid] (b2) to[out=90,in=270] node[left] {$X_2$} (b);
		\draw[mid] (b) to[out=90,in=240] node[left] {$Z_1$}(c);
		\draw[mid] (c) to[out=110,in=270] node[left] {$Y_2$} (t2);
		\draw[mid] (c) to[out=60,in=270] node[below] {$Y_3$}(t3);
		\draw[mid] (b3) to[out=90,in=270] node[left] {$X_3$} (d);
		\draw[mid] (d) to[out=100,in=290] node[right] {$Z_2$}(c);
		\draw[mid] (d) to[out=70,in=250] node[below] {$Z_3$}(e);
		\draw[mid] (b4) to[out=90,in=300] node[right] {$X_4$}(e);
        \draw (bl) -- (br);
        \draw (tl) -- (tr);
        \draw[dashed] (m1l) -- (m1r);
        \draw[dashed] (m2l) -- (m2r);
        \draw[dashed] (s1b) -- (s1t);
        \draw[dashed] (s2b) -- (s2t);
        \draw[dashed] (s3b) -- (s3t);
        \draw[dashed] (s4b) -- (s4t);
        \draw[dashed] (s5b) -- (s5t);
        \draw[dashed] (s6b) -- (s6t);
       
  \end{tikzpicture}
  }
  }
  \caption{Evaluation of $\Ccal$-colored progressive graph $\Gamma$ in $\Rbb\times [0,1]$.}
  \label{figure 1}
\end{figure}

The graph has ten edges, which are colored by the objects 
\linebreak $(X_1,X_2,X_3,X_4,Z_1,Z_2, Z_3,Y_1,Y_2,Y_3)$, and 13 nodes, 5 of which are inner nodes colored by morphisms $(f_1,f_2,f_3,f_4,f_5)$. It has domain $\domrm(\Gamma)=(X_1,\cdots,X_4)$ and codomain $\codomrm(\Gamma)=(Y_1,Y_2,Y_3)$. 
In addition to the graph, we show eight auxiliary dashed lines:
\begin{enumerate}[label=\arabic*)]
\item Two horizontal ones at $\Rbb\times \lbr t_1\rbr $ and $\Rbb\times \lbr t_2\rbr $. These are called \textit{regular level lines} and their levels $0<t_1<t_2<1$ are chosen such that $\Rbb\times \lbr t_i\rbr $ does not  intersect the inner nodes $\Gamma_0-\p \Gamma$. Cutting $\Gamma$ at $\Rbb\times \lbr t_1\rbr $ and $\Rbb\times \lbr t_2\rbr $, we get three consecutive progressive graphs $\Gamma_1$, $\Gamma_2$ and $\Gamma_3$, where $\Gamma_1$ is the progressive graph in $\Rbb\times [0,t_1]$, $\Gamma_2$ is the one in $\Rbb\times [t_1,t_2]$ and $\Gamma_3$ is the top one in $\Rbb\times [t_2,1]$.
\item Six vertical lines, three in $\Gamma_1$, two in $\Gamma_2$ and one in $\Gamma_3$. Each collection of vertical lines gives a \textit{tensor decomposition} of $\Gamma_1$, $\Gamma_2$ and $\Gamma_3$, respectively. E.g. the three vertical lines in $\Gamma_1$, split it into a disjoint union of four graphs $\Gamma_1^i$, $i=1,\cdots, 4$, which are linearly ordered from left to right. Each $\Gamma_1^i$ either contains exactly one inner node or does not contain an inner node. 
\end{enumerate}
The $\Ccal$-coloring of $\Gamma$ associates to $\Gamma_1^i$ a morphism in $\Ccal$. For the graphs $\Gamma_1^i$ these are
\eq{
f_{\Gamma_1^1}=\idrm_{X_1},\quad f_{\Gamma_1^2}=\idrm_{X_2},\quad f_{\Gamma_1^3}=f_4,\quad f_{\Gamma_1^4}=\idrm_{X_4}\, ,
}
with $f_4\in \hom_\Ccal(X_3,Z_2\otimes Z_3)$ as in figure \ref{figure 1}. 
The progressive graph $\Gamma_1$  thus evaluates to the morphism
\eq{
f_{\Gamma_1}\coloneqq f_{\Gamma_1^1}\otimes f_{\Gamma_1^2}\otimes f_{\Gamma_1^3}\otimes f_{\Gamma_1^4}:X_1\otimes X_2\otimes X_3\otimes X_4\rightarrow X_1\otimes X_2\otimes Z_2\otimes Z_3\otimes X_4, 
}
i.e.
$ f_{\Gamma_1}=\idrm_{X_1}\otimes \idrm_{X_2}\otimes f_4\otimes \idrm_{X_4}$.
The morphisms $f_{\Gamma_2}$ and $f_{\Gamma_3}$ are defined analogously. The morphism associated to the whole progressive graph is given by
 \eq{\label{morphism from graph}
 f_\Gamma\coloneqq f_{\Gamma_3}\circ f_{\Gamma_2}\circ f_{\Gamma_1} \,\,. 
 }
 \begin{rem}
We highlight the two very different roles of the $x$-direction and the $y$-directions 
in the plane: The horizontal $x$-direction corresponds to the monoidal product in $\Ccal$, whereas the vertical $y$-direction corresponds to the composition of morphisms. In other words, the implicitly chosen standard $2$-framing on the strip $\Rbb\times [0,1]$ is essential for evaluating a progressive graph $\Gamma$
to a  morphism in $\Ccal$.
 \end{rem}
 
By one of the main results in \cite{JOYAL199155}, morphism $f_\Gamma:\domrm(\Gamma,\varphi)\rightarrow \codomrm(\Gamma,\varphi)$
constructed for a $\Ccal$-colored progressive graph $\Gamma$ depends neither on the choice of the regular level lines, nor on the tensor decomposition. 
Consider two $\Ccal$-colored progressive graphs $(\Gamma_1,\varphi_1)$, $(\Gamma_2,\varphi_2)$ in $\Rbb\times[0,1]$. We say that \textit{$\Gamma_1$ and $\Gamma_2$ are progressively isotopic}, if there exists an isotopy $H:[0,1]\times (\Rbb\times [0,1])$ from $\Gamma_1$ to $\Gamma_2$ such that $H(s,\bullet)|_{\Gamma_1}$ is a progressive graph for all $s\in [0,1]$. The isotopy $H$ is called a \textit{progressive isotopy}. Invariance of the associated morphism for a $\Ccal$-colored progressive graph under the auxiliary decomposition in regular levels and tensor decompositions is then linked to the invariance under progressive isotopies, i.e. if $(\Gamma_1,\varphi_1)$ and $(\Gamma_2,\varphi_2)$ are progressively isotopic, then $f_{\Gamma_1}=f_{\Gamma_2}$. 

Conversely, every morphism in $\Ccal$ can be represented by a $\Ccal$-colored graph:

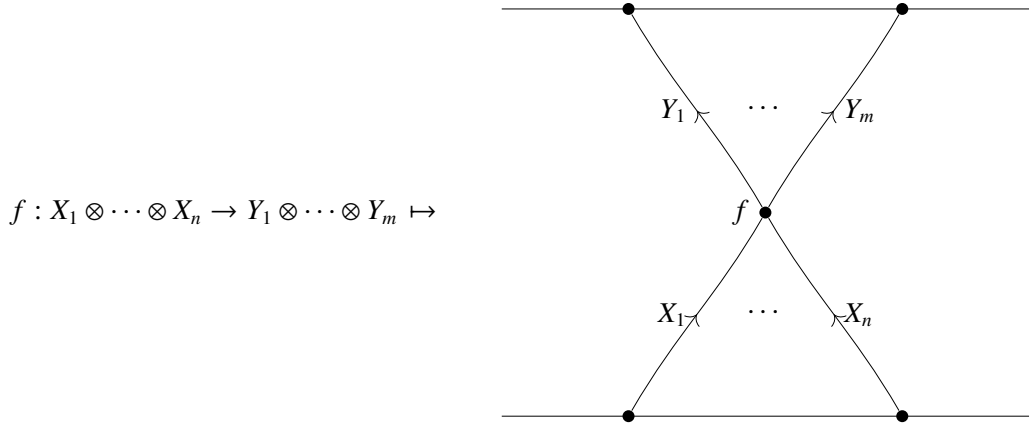
\begin{figure}[H]
    \makebox[\textwidth][c]{
    \resizebox{0.9\width}{!}{
    \begin{tikzpicture}
    \node at (0,0) (a) {$f:X_1\otimes\cdots\otimes X_n\rightarrow Y_1 \otimes \cdots \otimes Y_m\quad$};
    \node at ($(a)+(3,0)$) (b) {$\mapsto$};
    \node[dot, label=left:$f$] at ($(a)+(8,0)$) (c) {};
    \node[dot] at ($(c)+(-2,-3)$) (bl) {};
    \node[dot] at ($(bl)+(4,0)$) (br) {};
    \node at ($(c)+(0,-1.5)$) (dots) {$\cdots$};
    \node[dot] at ($(bl)+(0,6)$) (tl) {};
    \node[dot] at ($(tl)+(4,0)$) (tr) {};
    \node at ($(dots)+(0,3)$) (tdots) {$\cdots$};
    \node at ($(bl)+(-2,0)$) (bll) {};
    \node at ($(br)+(2,0)$) (brr) {};
    \node at ($(bll)+(0,6)$) (tll) {};
    \node at ($(tr)+(2,0)$) (trr) {};
    
    \draw[mid] (bl) to[out=60,in=240] node[left] {$X_1$} (c);
    \draw[mid] (br) to[out=120,in=300] node[right] {$X_n$} (c);
    \draw[mid] (c) to[out=120,in=300] node[left] {$Y_1$} (tl);
    \draw[mid] (c) to[out=60,in=240] node[right] {$Y_m$} (tr);
    \draw (bll) -- (brr);
    \draw (tll) -- (trr); 
    
    \end{tikzpicture}
    }
    }
\end{figure}

Obviously, a morphism can have different realizations as a progressive graph. The graph $\Gamma$ from figure \ref{figure 1} describing the morphism $f_\Gamma$ is topologically very different from the graph with a single inner node colored by $f_\Gamma$ in equation \eqref{morphism from graph}. As in the oriented case, identifying different graphical realizations of the same morphism will be at the heart of the framed string-net construction.

\section{Framed String-Net Construction} \label{framed construction}

In this section, we define string-nets on $2$-framed surfaces. The algebraic input for
our string-net construction is  a  tensor category; as output, it produces a 
vector space for any $2$-framed surface. 
The main point of the construction is to 
globalize the discussion of progressive graphs from the standard framed plane 
in section \ref{Graphical calculus for Finite Tensor Categories}
to an arbitrary framed surface. 

\subsection{Locally Progressive Graphs}\label{Locally Progressive Graphs}
\begin{defn} Let $\Sigma$ be a smooth surface. $\Sigma$ is \textit{$2$-framed} if there exist two nowhere-vanishing vector fields $X_1,X_2\in \Gamma(T\Sigma)$ such that $((X_1)_p,(X_2)_p)\in T_p\Sigma$ is an ordered basis for every $p\in \Sigma$. The pair $(X_1,X_2)$ is a \textit{global ordered frame} for the tangent bundle $T\Sigma$ of $\Sigma$.
\end{defn}

To any vector field $X$ on $\Sigma$, we can associate its \textit{maximal flow} $\theta:D\rightarrow \Sigma$. The domain is a subset $D\subset \Rbb\times \Sigma$ where $D^{(p)}\coloneqq \lbr t\in \Rbb\, \middle|\, (t,p)\in D\rbr $ is an open interval. $D$ is called a \textit{flow domain}. The flow $\theta$ satisfies $\theta(0,p)=p$ and $\theta(t_1,\theta(t_2,p))=\theta(t_1+t_2,p)$ for all $p\in \Sigma$. The flow is \textit{maximal for $X$} in the sense that for all $p\in \Sigma$, the curve
\eq{
\theta(\bullet,p)\rightarrow \Sigma
}
is the unique maximal integral curve of $X$, i.e. $\frac{\drm}{\drm t} \theta(t,p)=X_{\theta(t,p)}$ with initial value $\theta(0,p)=p$.
For a global frame $(X_1,X_2)$ on $\Sigma$, we denote by $\theta_1:D_1\rightarrow \Sigma$ and $\theta_2:D_2\rightarrow \Sigma$ the corresponding maximal flows. The maximal integral curves for $(X_1,X_2)$ through a point $p\in \Sigma$ are denoted by $\theta_1^{(p)}:D_1^{(p)}\rightarrow \Sigma$ and $\theta_2^{(p)}:D_2^{(p)}\rightarrow \Sigma$. Since $X_1,X_2$ are nowhere-vanishing, the curves $\theta_1^{(p)}$, $\theta_2^{(p)}$ are smooth immersions for all $p\in \Sigma$. Further details on maximal flows and framed manifolds and flows can be found e.g. in \cite[Chapter~9]{lee2013smooth}.

Recall that a planar graph was defined as an abstract graph $(\Gamma,\Gamma_0, \p \Gamma)$ with a smooth map $\iota:\Gamma \rightarrow \Rbb^2$ such that $\iota|_{\Gamma-\Gamma_0}$ is a smooth embedding. Similarly, for $(\Sigma,\p\Sigma)$ a smooth surface $\Sigma$ with boundary $\p\Sigma$ an \textit{embedded graph} is an abstract graph $(\Gamma, \Gamma_0,\p\Gamma)$ together with a smooth map $\iota_\Sigma:\Gamma\rightarrow \Sigma$ such that $\iota_\Sigma|_{\Gamma-\Gamma_0}$ is an embedding and $\iota_\Sigma(\p\Gamma)=\iota_\Sigma(\Gamma)\cap \p\Sigma$. 
For an embedded graph $(\Gamma,\iota_\Sigma)$, we usually suppress the embedding $\iota_\Sigma$ from the notation. 

We want to formulate the equivalent of a progressive graph for an arbitrary $2$-framed
surface. To do so, we have to generalize the condition of injectivity of the projection to the second component that features in
the definition of a progressive graph. 
The idea is to formulate a local condition on graphs at every point on the surface. Using the global frame of a $2$-framed surface $\Sigma$, there is a neighborhood around every $p\in \Sigma$, which looks like the strip $\Rbb\times [0,1]$ and the two vector fields give the two distinguished directions on the strip. The flow lines of $X_2$ are then a natural analog of the vertical $y$-direction in the plane and we can perform a projection to $X_2$-flow lines by moving points along the flow of $X_1$ (see figure \ref{locally progressive graph}). Given an embedded graph $\Gamma\subset \Sigma$, we require that locally around every point, this projection, restricted to
$\Gamma$,  is injective. This allows us to define a local evaluation map of an embedded $\Ccal$-graph, which is the framed analog of the evaluation of graphs inside of disks in the oriented case.

A variant of the flow-out theorem \cite[Theorem 9.20]{lee2013smooth} shows that for a $2$-framed surface $\Sigma$ with global frame $(X_1,X_2)$ and corresponding flow domains $D_1$, $D_2$, for every point $p\in \Sigma$, there exist open intervals $I_1^{(p)}\subset D_1^{(p)}$, $I_2^{(p)}\subset D_2^{(p)}$ containing $0$ such that
\eq{
\phi^{(p)}:\ov{I}_1^{(p)}\times I^{(p)}_2&\hookrightarrow \Sigma\\ 
 (s,t)&\mapsto \theta_1(s,\theta_2(t,p))
 }
is a smooth embedding. Let $(\Gamma,\Gamma_0,\p\Gamma)$ be an embedded graph in $\Sigma$. An element $t\in I_2^{(p)}$ is \textit{regular} with respect to
$\Gamma$, if $\phi^{(p)}(I_1^{(p)} \times \{t\})\cap (\Gamma_0-\p\Gamma)=\emptyset$, i.e. the flow line of $X_1$ at $t$ inside $\phi^{(p)}(\ov{I}_1^{(p)}\times I_2^{(p)})$ does not contain any inner nodes of $\Gamma$. If $t_1<0<t_2$ are regular levels, the image $\phi^{(p)}(I_1^{(p)}\times [t_1,t_2])$ is called a \textit{standard rectangle} for $\Gamma$ at $p$. The restriction of $\Gamma$ to a standard rectangle at $p$ is denoted by $(\Gamma^{(p)}[t_1,t_2],\Gamma_0^{(p)}[t_1,t_2],\p\Gamma^{(p)}[t_1,t_2])$.

\begin{defn} Let $(\Sigma, (X_1,X_2))$ be a $2$-framed surface and $(\Gamma,\Gamma_0,\p\Gamma)$ an embedded graph in $\Sigma$. Then $\Gamma$ is a \textit{locally progressive graph} if for every $p\in \Sigma$, there exists a standard rectangle $\phi^{(p)}(I_1^{(p)}\times [t_1,t_2])$ for $\Gamma$ at $p$ such that the restriction of 
\eq{
\prrm^{(p)}_2\coloneqq \prrm_2\circ \left(\phi^{(p)}\right)^{-1}: \phi^{(p)}(I_1^{(p)}\times [t_1,t_2])&\rightarrow [t_1,t_2]\\
\phi^{(p)}(s,t)&\mapsto t
}
to $\Gamma^{(p)}[t_1,t_2]-\Gamma_0^{(p)}[t_1,t_2]$ is injective.
\end{defn}

To understand these definitions, it is best to consider figure \ref{locally progressive graph}. The figure shows a small patch of a $2$-framed surface $(\Sigma,(X_1,X_2))$. The red horizontal lines are flow lines of $X_1$ and the blue vertical line is a flow line of $X_2$. In black, we show an embedded graph. Each of the dashed horizontal lines intersects an edge of the embedded graph at a unique point. Transporting this intersection point along the horizontal line until we hit the vertical blue line, defines the projection map $\prrm_2^{(p)}$ evaluated at the intersection point. For the graph shown in figure \ref{locally progressive graph} the projection is obviously injective and thus, this is a locally progressive graph for the underlying
$2$-framed surface.

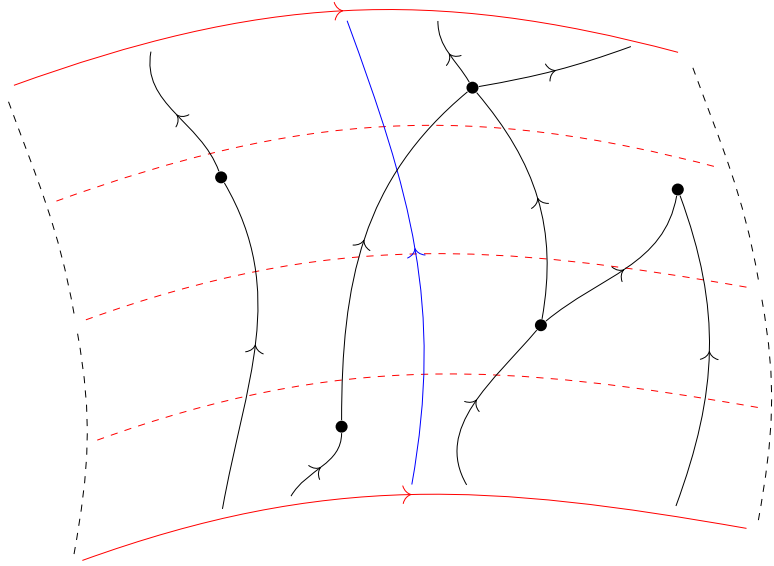
\begin{figure}
    \makebox[\textwidth][c]{
    \resizebox{0.9\width}{!}{
    \begin{tikzpicture}
    
    \node at (0,0) (tl) {};
    \node at ($(tl)+(10,0.5)$) (tr) {};
    \node at ($(tl)+(1,-7)$) (bl) {};
    \node at ($(bl)+(10,0.5)$) (br) {};
    
    \draw[mid, color=red] (tl) to[out=20,in=165] node[pos=0.2](a){} node[pos=0.5](mt) {} node[pos=0.65](b){} node[pos=0.95](c){} (tr);
    \draw[mid, color=red] (bl) to[out=20,in=170] node[pos=0.2](d){} node[pos=0.3](e){} node[pos=0.5](mb){} node[pos=0.6] (f) {} node[pos=0.9] (g) {}(br);
    \draw[color=white] (bl) to[out=80,in=290] node (ml) {} (tl);
    \draw[color=white] (br) to[out=80,in=290] node (mr) {} (tr);
    
    \draw[dashed] (bl) to[out=80,in=280] node (mbl) {} (ml);
    \draw[dashed] (ml) to[out=100,in=290] node (mtl) {} (tl);
    \draw[dashed] (br) to[out=80,in=280] node (mbr) {} (mr);
    \draw[dashed] (mr) to[out=100,in=290] node (mtr) {} (tr);
    
    \draw[dashed, color=red] (ml) to[out=20,in=170] (mr);
    \draw[dashed, color=red] (mtl) to[out=20,in=165] (mtr);
    \draw[dashed, color=red] (mbl) to[out=20,in=170] (mbr);
    \draw[mid, color=blue] (mb) to[out=80,in=290] (mt);
    
    
    \node[dot] at ($(a)+(1,-2)$) (g1) {};
    \node[dot] at ($(mb)+(-1,1)$) (g2) {};
    \node[dot] at ($(f)+(1,2.5)$) (g3) {};
    \node[dot] at ($(g3)+(2,2)$) (g4) {};
    \node[dot] at ($(g3)+(-1,3.5)$) (g5) {};
    
    \draw[mid] (d) to[out=80,in=300] (g1);
    \draw[mid] (g1) to[out=110,in=260] (a);
    \draw[mid] (e) to[out=60,in=270] (g2);
    \draw[mid] (g2) to[out=90,in=220] (g5);
    \draw[mid] (f) to[out=120,in=230] (g3);
    \draw[mid] (g3) to[out=40,in=260] (g4);
    \draw[mid] (g3) to[out=80,in=310] (g5);
    \draw[mid] (g5) to[out=120,in=270] (b);
    \draw[mid] (g5) to[out=10,in=200] (c);
    \draw[mid] (g) to[out=70,in=290] (g4);
    \end{tikzpicture}
    }
    }
    \caption{In the colored version, red horizontal lines correspond to flow lines of $X_1$ and the blue vertical line is a flow line of $X_2$. Together they yield a standard
    rectangle (and even an evaluation rectangle) for the locally progressive graph 
    shown in black.}
    \label{locally progressive graph}
    \end{figure}
\begin{defn}
Let $(\Gamma,\Gamma_0,\p\Gamma)$ be an embedded graph inside a framed surface $\Sigma$ and $\phi^{(p)}:\ov{I}_1^{(p)}\times I_2^{(p)}\hookrightarrow \Sigma$ a standard rectangle at $p$. Given two regular levels $t_1<0<t_2$ and $\left[s_1,s_2\right]\subset \ov{I}_1^{(p)}$, the image $\phi^{(p)}\left([s_1,s_2]\times [t_1,t_2]\right)$ is an \textit{evaluation rectangle} for $\Gamma$ at $p$, if 
\eq{
\Gamma\cap \phi^{(p)}(\lbr s_1,s_2\rbr\times I^{(p)}_2)=\emptyset
}
and
\eq{\label{no node intersection}
\Gamma_0\cap\phi^{(p)}\left([s_1,s_2]\times \lbr t_1,t_2\rbr\right) =\emptyset \,\,. 
}
Let now $\Ccal$ be again a  tensor category, which is not assumed to be pivotal. An evaluation rectangle at $p\in \Sigma$ for a $\Ccal$-colored graph $\Gamma$ will be denoted by $R_\Gamma^{(p)}$.
\end{defn}

Given an evaluation rectangle $R^{(p)}_\Gamma=\phi^{(p)}([s_1,s_2]\times [t_1,t_2])$ for a locally progressive graph $\Ccal$-colored graph $\Gamma$ in $\Sigma$, by \eqref{no node intersection}, only the lower and upper horizontal flow lines $\phi^{(p)}([s_1,s_2]\times {t_1})$, $\phi^{(p)}([s_1,s_2]\times {t_2})$ 
intersect edges of the graph $\Gamma$. 
We associate to each intersection point the corresponding $\Ccal$-color of the edge of $\Gamma$. Taking the tensor product of these elements according to the linear order on $[s_1,s_2]$ gives the \textit{(co-)domain of $\Gamma$ with respect to $R_\Gamma^{(p)}$}, which will be denoted by $\domrm_R(\Gamma)$ and $\codomrm_R(\Gamma)$, respectively. Note that in analogy to the (co-)domain of a progressive graph, we have
$\domrm_R(\Gamma)$, $\codomrm_R(\Gamma)\in \mathrm{ob}(\Ccal)$.

\begin{rem}\label{preimage is evaluation} From the definition of a locally progressive graph, it directly follows that the preimage of $\Gamma$ is a progressive graph in
the rectangle	
$[s_1,s_2]\times [t_1,t_2]$ for every evaluation rectangle $\phi^{(p)}([s_1,s_2]\times [t_1,t_2])$. The $\Ccal$-colored progressive graph has (co-)domain $\mathrm{(co-)dom}_R(\Gamma)$ and yields a morphism in $f_R^\Gamma\in\hom_\Ccal(\domrm_R(\Gamma),\codomrm_R(\Gamma))$. This defines an evaluation map $\nu_R(\Gamma)\coloneqq f^\Gamma_R$.
\end{rem}

\begin{rem} When defining the evaluation of a $\Ccal$-colored progressive graph, we stressed the very different roles the $x-$ and $y$-directions had in the plane. The first corresponds to taking tensor products in $\Ccal$, whereas the latter encodes the composition of morphisms. The vector fields of a global frame have similar roles for $\Ccal$-colored embedded graphs. As stated in Remark \ref{preimage is evaluation}, the $y$-flow lines define domain and codomain for the morphism corresponding to a locally progressive graph, whereas going along $x$-flow lines corresponds to taking tensor products. 
\end{rem}

\subsection{Framed String-Net Spaces}\label{Framed String-Net Spaces}

Let $\Ccal$ be a tensor category and $\Sigma$ a $2$-framed surface. We 
now define a string-net space in terms of $\Ccal$-graphs on $\Sigma$, which 
we are going to call framed string-net space. 

\begin{defn} Let $B\coloneqq \lbr p_1,\cdots, p_n\rbr \subset \p\Sigma$ be a finite and possibly empty subset of the boundary of the surface $\Sigma$ and $\nu_B:B\rightarrow \mathrm{ob}(\Ccal)$ a map. The pair $(B,\nu_B)$ is called a \textit{boundary value}.
\end{defn}

Let $(\Gamma,\Gamma_0,\p\Gamma)$ be $\Ccal$-colored embedded graph in $\Sigma$. Boundary nodes of $\Gamma$ are mapped to the boundary $\p\Sigma$ of the surface. This gives a finite subset $B_\Gamma$ of the boundary. Defining a map $\nu_\Gamma:B_\Gamma\rightarrow \mathrm{ob}(\Ccal)$ by mapping the boundary node to the $\Ccal$-color of its adjacent edge, we obtain a boundary value $(B_\Gamma,\nu_\Gamma)$ for a $\Ccal$-colored embedded graph. We call this the \textit{boundary value} of the graph $\Gamma$.

\begin{defn} The set of all $\Ccal$-colored locally progressive graphs on a $2$-framed surface $\Sigma$ with boundary value $(B,\nu_B)$ is denoted by
\eq{
\Graphrm(\Sigma,(B,\nu_B))\,\, .
}
The vector space 
\eq{
\VGraphrm_\Kbb(\Sigma,(B,\nu_B))\coloneqq \spanrm_\Kbb \Graphrm(\Sigma,(B,\nu_B))
}
freely generated by this set is called \textit{framed pre-string-net space}.
\end{defn}

From now on all string-nets on $2$-framed surfaces will be locally progressive. Similar to the construction of string-net spaces on oriented surfaces, we want to identify elements of $\VGraphrm(\Sigma,(B,\nu_B))$ if they locally evaluate to the same morphism in $\Ccal$. However, the additional datum of a $2$-framing on $\Sigma$ allows us to use evaluation rectangles of graphs instead of disks so that as an algebraic input we do not need a pivotal structure on $\Ccal$. By Remark \ref{preimage is evaluation} the preimage of a locally progressive graph inside every evaluation rectangle is a progressive graph. Thus, we can use the evaluation map for $\Ccal$-colored progressive graphs we explained in section \ref{Graphical calculus for Finite Tensor Categories}
to associate to every $\Ccal$-colored locally progressive graph and evaluation rectangle $\phi^{(p)}\left([s_1,s_2]\times [t_1,t_2]\right)$ at any point $p\in \Sigma$ a morphism in $\Ccal$.

\begin{defn}\label{null graphs}

Let $(B,\nu_B)$ be a boundary value and $\Gamma_1,\cdots,\Gamma_n\in \Graphrm(\Sigma,(B,\nu_B))$. For $\lambda_1,\cdots, \lambda_n\in \Kbb$, the element $\Gamma\coloneqq\sum_{i=1}^{n}\lambda_i\Gamma_i\in \VGraphrm_\Kbb(\Sigma,(B,\nu_B))$ is a \textit{null graph}, if there exists a common evaluation rectangle $R^{(p)}\coloneqq\phi^{(p)}\left([s_1,s_2]\times [t_1,t_2]\right) $ for all $\Gamma_i$ such that
\begin{enumerate}[label=\roman*)]
\item \eq{ \Gamma_i\cap \phi^{(p)}([s_1,s_2]\times \lbr t_1,t_2\rbr) =\Gamma_j\cap \phi^{(p)}([s_1,s_2]\times \lbr t_1,t_2\rbr)
}
for all $i,j=1,\cdots, n$.
\item $\domrm_R(\Gamma)\coloneqq\domrm_R(\Gamma_i)=\domrm_R(\Gamma_j)$ and $\codomrm_R(\Gamma)\coloneqq\codomrm_R(\Gamma_i)=\codomrm(\Gamma_j)$ for all $i,j=1,\cdots, n$. 
\item $\Gamma_i|_{\Sigma-R^{(p)}}=\Gamma_j|_{\Sigma-R^{(p)}}$ for all $i,j=1,\cdots, n$.
\item 
\eq{
\sum_{i=1}^n\lambda_i\nu_R(\Gamma_i)=0\in \hom_\Ccal(\domrm_R(\Gamma),\codomrm_R(\Gamma))\,\,.
 }
\end{enumerate}

The sub-vector space spanned by all null graphs is denoted by $\NGraphrm(\Sigma,(B,\nu_B))$.
\end{defn}

\begin{defn}\label{framed string-net space} Let $\Sigma$ be a framed surface, $\Ccal$ a tensor category and $(B,\nu_B)$ be a boundary value in $\Ccal$. The \textit{framed string-net space} with boundary value $(B,\nu_B)$ is defined as the vector space quotient
\eq{
\SNrm^{fr}(\Sigma,(B,\nu_B))\coloneqq \frac{\VGraphrm(\Sigma,(B,\nu_B))}{\NGraphrm(\Sigma,(B,\nu_B))}\,\,.
}
\end{defn}

\begin{rem}
Taking the quotient by null graphs also takes appropriate isotopies between locally progressive graphs into account. Recall that we defined locally progressive graphs as embedded graphs with a fixed embedding. Thus, a priori abstract $\Ccal$-colored graphs with different embeddings yield different elements in $\VGraphrm(\Sigma)$. By taking the above quotient, we can identify embedded graphs that differ by those isotopies such that graphs along the isotopy are all locally progressive graphs. 
\end{rem}

\section{Circle Categories and Twisted Drinfeld Centers}\label{cylinder section}

In this final section, we put our construction of string-nets for framed
surfaces to the test and compute the relevant circle categories. We show
that they are related to Drinfeld centers twisted by appropriate powers of
the double dual.

\subsection{$2$-Framings of the Circle and Framed Cylinders}\label{$2$-Framings of the Circle and Framed Cylinders}
A \textit{$2$-framing} of a circle $S^1$ is an isomorphism $\lambda:TS^1\oplus\ul{\Rbb}\xrightarrow{\simeq}\ul{\Rbb^2}$ of vector bundles, where $\ul{\Rbb}\rightarrow S^1$ and $\ul{\Rbb^2}\rightarrow S^1$ are the trivial vector bundles with fibers $\Rbb$ and $\Rbb^2$, respectively. There is a bijection \cite[section~1.1]{douglas2020dualizable}
\eq{
\lbr \text{Homotopy classes of $2$-framings of }S^1\rbr \simeq \Zbb \,\,.
}

The different $2$-framings for $n\in \Zbb$ can be depicted as follows. We identify $S^1$ as the quotient $S^1\simeq [0,1]/ 0\sim 1$ and draw a circle as an interval, while keeping in mind that we identify the endpoints. The integer $n$ then counts the number of full rotations in counterclockwise direction a frame of $\Rbb^2$ undergoes while going around the circle. We denote the circle with $2$-framing corresponding to $n\in \Zbb$ by $S^1_n$. 
We can trivially continue the $2$-framing of $S^1_n$ along the radial direction of a cylinder over $S^1$. This gives a $2$-framed cylinder $\Csf$ over
the circle, i.e.\ an annulus with a distinguished radial direction, which can be seen as $2$-framed cobordism $\Csf:S^1_n\rightarrow S^1_n$.  Possibly after a global 
rotation of the two vector fields, we can arrange that there is at least one point on
$S^1$ such that the flow line for the second vector field is radial. We fix such
a point as an auxiliary datum and call the corresponding flow line the
\textit{distinguished radial line}.

We denote the cylinder with this particular $2$-framing corresponding to $n\in\Zbb$ by $\Csf_n$. The flow lines for $\Csf_{-1}$, $\Csf_0$ and $\Csf_1$ are shown in figure \ref{different framings}.
\begin{figure}
\centering
\scalebox{0.5}{
\begin{tikzpicture}
	\draw (0,0) circle (5.5cm);
	\clip[draw] circle (5.5cm);
	\draw[postaction={decorate, decoration={
			markings, mark= between positions 0cm and 1 step 1cm with {\arrow{<[scale=1.5]}}}},blue] (0,5.5) -- (0,0);
	\foreach \x in {0,0.5,...,3}{
		\draw[postaction={decorate, decoration={
				markings, mark= between positions 0.2 and 0.8 step 0.2 with {\arrow{<[scale=1.5]}}}},blue] (0,0) arc (360:0:\x);
	}
	\draw[postaction={decorate, decoration={
			markings, mark= between positions 0.2 and 0.8 step 0.2 with {\arrow{<[scale=1.5]}}}},blue] (0,0) arc (360:0:4.5);
	\draw[postaction={decorate, decoration={
			markings, mark= between positions 0.2 and 0.8 step 0.2 with {\arrow{<[scale=1.5]}}}},blue] (0,0) arc (360:0:7.5);
	
	\draw[postaction={decorate, decoration={
			markings, mark= between positions 0.5cm and 1 step 1cm with {\arrow{<[scale=1.5]}}}},blue] (0,0) -- (0,-5.5);
	
	\foreach \x in {-3,-2.5,...,0}{
		\draw[postaction={decorate, decoration={
				markings, mark= between positions 0.2 and 0.8 step 0.2 with {\arrow{<[scale=1.5]}}}},blue] (0,0) arc (0:360:\x);
	}
	\draw[postaction={decorate, decoration={
			markings, mark= between positions 0.2 and 0.8 step 0.2 with {\arrow{<[scale=1.5]}}}},blue] (0,0) arc (0:360:-4.5);
	\draw[postaction={decorate, decoration={
			markings, mark= between positions 0.2 and 0.8 step 0.2 with {\arrow{<[scale=1.5]}}}},blue] (0,0) arc (0:360:-7.5);
	
	\draw[postaction={decorate, decoration={
			markings, mark= between positions 0.5cm and 1 step 1cm with {\arrow{<[scale=1.5]}}}},red]  (0,0) -- (5.5,0);
	\foreach \y in {0,0.5,...,3}{
		\draw[postaction={decorate, decoration={
				markings, mark= between positions 0.2 and 0.8 step 0.2 with {\arrow{<[scale=1.5]}}}},red] (0,0) arc (270:630:\y);
	}
	\draw[postaction={decorate, decoration={
			markings, mark= between positions 0.2 and 0.8 step 0.2 with {\arrow{<[scale=1.5]}}}},red] (0,0) arc (270:630:4.5);
	\draw[postaction={decorate, decoration={
			markings, mark= between positions 0.2 and 0.8 step 0.2 with {\arrow{<[scale=1.5]}}}},red] (0,0) arc (270:630:7.5);
	
	\draw[postaction={decorate, decoration={
			markings, mark= between positions 0cm and 1 step 1cm with {\arrow{<[scale=1.5]}}}},red] (-5.5,0) -- (0,0);
	\foreach \y in {0,0.5,...,3}{
		\draw[postaction={decorate, decoration={
				markings, mark= between positions 0.2 and 0.8 step 0.2 with {\arrow{<[scale=1.5]}}}},red] (0,0) arc (90:-270:\y);
	}
	\draw[postaction={decorate, decoration={
			markings, mark= between positions 0.2 and 0.8 step 0.2 with {\arrow{<[scale=1.5]}}}},red] (0,0) arc (90:-270:4.5);
	\draw[postaction={decorate, decoration={
			markings, mark= between positions 0.2 and 0.8 step 0.2 with {\arrow{<[scale=1.5]}}}},red] (0,0) arc (90:-270:7.5);
	\draw[fill=white] (0,0) circle (0.5cm);
\end{tikzpicture}}
\scalebox{0.5}{
\begin{tikzpicture}
		\draw (0,0) circle (5.5cm);
		\clip[draw] circle (5.5cm);
		\foreach \x in {0,15,...,345}{
			\draw[postaction={decorate, decoration={
					markings, mark= between positions 1.5cm and 4.5cm step 1cm with {\arrow{>[scale=1.5]}}}}, blue, rotate=\x] (0,0) -- (5.5, 5.5);
		}
		\foreach \y in {0,1,...,5}{
			\draw[postaction={decorate, decoration={
					markings, mark= between positions 0 and 1 step 0.2 with {\arrow{>[scale=1.5]}}}},red] (0,\y) arc (90:-270:\y);
		}
		\draw[fill=white] (0,0) circle (0.5cm);
\end{tikzpicture}}
\quad
\scalebox{0.5}{
\begin{tikzpicture}[decoration={
			markings,
			mark= between positions 0 and 1 step 1cm with {\arrow{>[scale=1.5]}}}
		]
		\draw (0,0) circle (5.5cm);
		\clip[draw] circle (5.5cm);
		\foreach \x in {-5,-4,...,5}{
			\draw[postaction={decorate},blue] (\x,-5.5) -- (\x, 5.5);
		}
		\foreach \y in {-5,-4,...,5}{
			\draw[postaction={decorate},red] (-5.5,\y) -- (5.5,\y);
		}
		\draw[fill=white] (0,0) circle (0.5cm);
\end{tikzpicture}}
\caption{Flow lines for the framed cylinders $\Csf_{-1}$, $\Csf_0$ and $\Csf_1$.}
\label{different framings}
\end{figure}
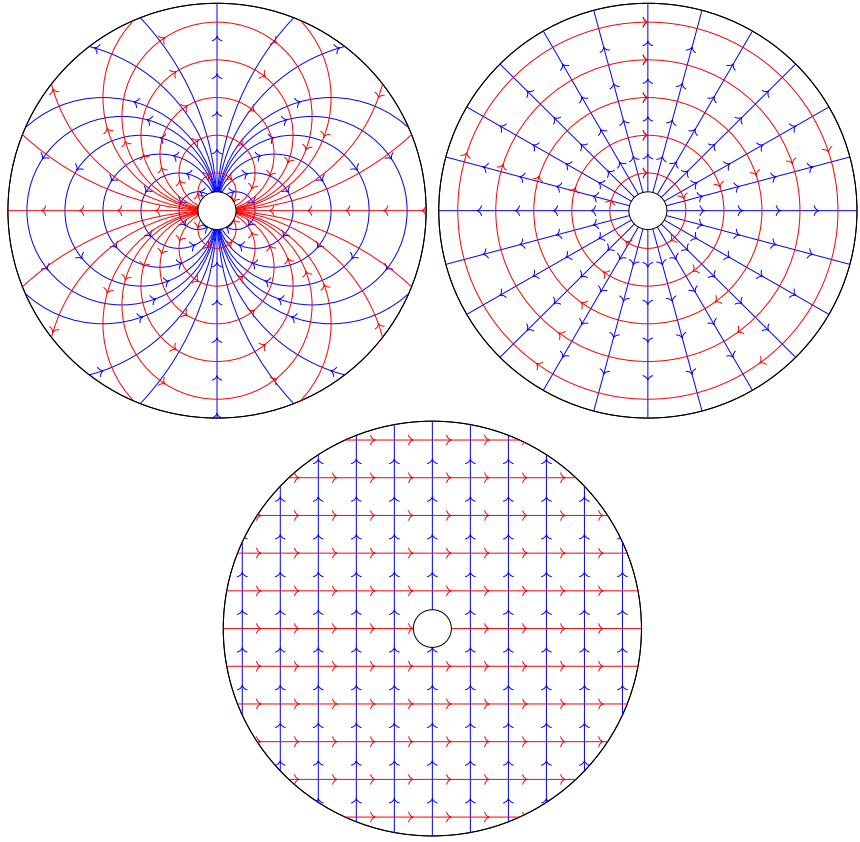

\subsection{Circle Categories}\label{Circle Categories}

Given a finite tensor category $\Ccal$ and a $2$-framed cylinder $\Csf_n$ over a one-manifold, we construct a $\Vectsf_\Kbb$-enriched category as follows. 
\begin{defn}\label{circle category}
The \textit{circle category} $\Cylsf(\Csf_n,\Ccal)$ is defined as follows:
\begin{itemize}
\item the \textit{objects} of $\Cylsf(\Csf_n,\Ccal)$ are the objects of $\Ccal$;
\item the vector space of \textit{morphisms} between two objects 
$X,Y\in \Cylsf(\Csf_n,\Ccal)$ is the framed string-net space
\eq{
\hom_{\Cylsf(\Csf_n,\Ccal)}(X,Y)\coloneqq \SNrm^{fr}(\Csf_n,B_{X,Y})
}
where we take the boundary value $B_{X,Y}\coloneqq (\lbr  p_1,p_2\rbr ,(X,Y))$ with the chosen point $p_1$ on $S^1\times \lbr 0\rbr $ and its counterpart $p_2$ on $S^1\times \lbr 1\rbr $ in $\Csf_n$.
\end{itemize}
The composition of morphisms is given by stacking cylinders and concatenating the corresponding string-nets.
\end{defn}

For the related notion of a tube category, we refer to 
\cite{haki}.

We first define a functor $I:\,\,\Ccal\to \Cylsf(\Csf_n,\Ccal)$  which is the
identity on objects. It maps a morphism $f: c_1 \to c_2$ in $\Ccal$
to the string-net which has two edges, both on the distinguished radial line,
with a single node on this line, labeled by $f$. 

In the following, we consider as an example
the \textit{blackboard framed cylinder} which is the framed surface $\Csf_1$ in figure \ref{different framings}.

\subsection{Circle Categories as Kleisli Categories}\label{Cylinder Kleisli}

To describe the morphism spaces of the circle category purely in terms of algebraic data, we need to know that string-net constructions obey factorization. This
has been discussed repeatedly in the literature, starting from
\cite[Section 4.4]{walk4}. Other references include \cite[p.\ 40]{hoek}
and \cite[Section 7]{kiTha}. The idea is that gluing relates the left exact 
functors associated to a surface to a coend. The cylinder can be obtained by gluing
a rectangle at two opposite boundaries; taking the insertions at the remaining
boundaries into account and using the fact that for the rectangle string-net spaces
give morphisms in $\Ccal$, the idea to implement factorization by a coend yields
\eq{
	\hom_{\Cylsf(\Csf_1,\Ccal)}(\bullet,\bullet)
	\cong\oint^{c\in\Ccal}\hom_\Ccal(\,(\bullet)\,,c\otimes(\bullet)\otimes \prescript{\vee}{}{c})\,\, .
\label{walker fact}	
}

\begin{lem}\label{cylinder left exact coend}
Let $X$, $Y\in \Ccal$ be two objects of a finite tensor category $\Ccal$. Then there is an isomorphism of vector spaces
\eq{
\hom_{\Cylsf(\Csf_1,\Ccal)}(x,y)\simeq \hom_\Ccal(x,Ty)
}
where $T\coloneqq {}_{\idrm}T_\idrm$ is the usual central monad of $\Ccal$.
\end{lem}
\begin{proof}
Recall from Lemma \ref{Kleisli left exact coend} that 
\eq{
\hom_\Ccal(x,Ty)=\oint^{c\in\Ccal}\hom_\Ccal(\,(\bullet)\,,c\otimes(\bullet)\otimes \prescript{\vee}{}{c})(x,y)\,\, .
}
and combine it with the factorization (\ref{walker fact}).
\end{proof}

\begin{theo}\label{main theo}
There is an equivalence of $\Vectsf_\Kbb$-enriched categories 
\eq{
\Cylsf(\Csf_1,\Ccal) \cong \Ccal_T\, .
}
\end{theo}
\begin{proof}
Note that the circle category $\Cylsf(\Csf_1,\Ccal)$ and the Kleisli category $\Ccal_T$ have the same objects as $\Ccal$. Thus we can define a functor 
\eq{
\kappa:\Cylsf(\Csf_1,\Ccal)\rightarrow  \Ccal_T
}
which is the identity on objects and acts on morphism spaces via the isomorphism induced by Lemma \ref{cylinder left exact coend}. For $\kappa$ to be a functor, we need to check that it respects identity morphisms and composition of morphisms. For $\ov{x}$, $\ov{y}\in \Ccal_T$, it holds that $\hom_{\Ccal_T}(\ov{x},\ov{y})=\hom_\Ccal(x,Ty)$. Let $\lbr \iota_c : c\otimes (\bullet) \otimes \prescript{\vee}{}{c}\Rightarrow T(\bullet)\rbr_{c\in \Ccal}$ be the universal dinatural family for the coend $T$. Then $\lbr (\iota_c)_\ast:\hom_\Ccal((\bullet),c\otimes (\bullet)\otimes \prescript{\vee}{}{c})\Rightarrow \hom_\Ccal((\bullet),T(\bullet))\rbr_{c\in \Ccal}$ is the universal dinatural family for the left exact coend $\hom_\Ccal((\bullet),T(\bullet))\simeq \oint^{c\in \Ccal}\hom_\Ccal((\bullet),c\otimes (\bullet)\otimes \prescript{\vee}{}{c})$. From the proof of Lemma \ref{cylinder left exact coend}, we get that $\kappa$ maps a string-net of the following form as

\eq{
\cylinder{
		\point{1}(0,-0.5)({},north)
		\point{0}(0,-5.5)({},north)
		\point{A}(0,-3)($h$,north east)
		\point{B}(0,3)($\evrm_c$,south)
		\line(A,B)(150,210,1.5)($c$,east)
		\line(A,B)(30,-30,1.5)($\ld{c}$,west)
		\line(0,A)(90,270,1)($x$,west)
		\line(A,1)(90,270,1)($y$,west)
	} \mapsto (\iota_c)_y \circ h \in \hom_{\Ccal}(x,Ty)\,\,.
}
For the identity in $\hom_{\Cylsf(\Csf_1,\Ccal)}(x,x)$, we get
\eq{
\cylinder{
		\point{1}(0,-0.5)({},north)
		\point{0}(0,-5.5)({},north)
		\point{A}(0,-3)($\idrm$,north east)
		\line(0,A)(90,270,1)($x$,west)
		\line(A,1)(90,270,1)($x$,west)
	} \mapsto (\iota_\onebb)_x \circ \idrm_{x} \in \hom_{\Ccal}(x,Tx)\,\,.
}
The morphism $\iota_\onebb: x \rightarrow Tx$ is the unit of the monad $T$ and thus corresponds to the identity morphism in $\hom_{\Ccal_T}(\ov{X},\ov{X})$. Composing two string-nets on $\Csf_1$ in standard form, we get 
\eq{
\cylinder{
		\point{1}(0,-0.5)({},north)
		\point{0}(0,-5.5)({},north)
		\point{A1}(0,-3+0.75)($g$,north east)
		\point{A2}(0,-3-0.75)($h$,north east)
		\point{B1}(0,3-0.75)($\evrm_d$,south)
		\point{B2}(0,3+0.75)($\evrm_c$,south)
		\line(A1,B1)(150,210,1.5)($d$,east)
		\line(A1,B1)(30,-30,1.5)($\ld{d}$,west)
		\line(A2,B2)(150,210,1.5)($c$,east)
		\line(A2,B2)(30,-30,1.5)($\ld{c}$,west)
		\line(A1,1)(90,270,1)($z$,west)
		\line(A2,A1)(90,270,1)($y$,west)
		\line(0,A2)(90,270,1)($x$,west)
	}
    & =
	\cylinder{
		\point{1}(0,-0.5)({},north)
		\point{0}(0,-5.5)({},north)
		\point{A}(0,-3)($(\idrm \otimes g \otimes \idrm) \circ h$,north east)
		\point{B}(0,3)($\evrm_{c \otimes d}$,south)
		\line(A,B)(150,210,1.5)($c \otimes d$,east)
		\line(A,B)(30,-30,1.5)($\ld{(c \otimes d)}$,west)
		\line(0,A)(90,270,1)($x$,west)
		\line(A,1)(90,270,1)($z$,west)
	}
\\
  &  \mapsto (\iota_{(c \otimes d)})_z \circ (\idrm \otimes g \otimes \idrm) \circ h\,\,.
}
There is the commutative diagram
\begin{center}
\begin{tikzcd}
x \ar[r,"h"]\ar[dr,bend right, "(\iota_c)_y \circ h"'] & c \otimes y \otimes \prescript{\vee}{}{c}\ar[r,"\idrm\otimes g\otimes \idrm"]\ar[d,"(\iota_c)_y"] & c\otimes d\otimes z\otimes \prescript{\vee}{}{d} \otimes \prescript{\vee}{}{c} \ar[dr, bend left, "(\iota_{c\otimes d})_z"]\ar[d,"(\iota_c)_{T(z)} \circ(\idrm \otimes (\iota_d)_z \otimes \idrm)"] & \\
& Ty \ar[r,"T((\iota_d)_z \circ g)"] & T^2(z)\ar[r,"\mu_z"] & Tz.
\end{tikzcd}
\end{center}

The lower path is the composition $(\alpha_d \circ g)\circ_{\Ccal_T} (\alpha_c \circ h)$ in $\Ccal_T$. By Lemma \ref{cylinder left exact coend}, $\kappa$ is fully faithful and since it is essentially surjective, it is an equivalence.
\end{proof}

Recall the functor $I:\,\,\Ccal\to \Cylsf(\Csf_1,\Ccal)$  introduced
at the end of section \ref{Circle Categories}. Under the equivalence between
the circle category and the Kleisli category, it is mapped to the induction
functor $I_T:\,\,\Ccal\to \Ccal_T$. Combining 
from Theorem \ref{main theo}, Proposition \ref{T-mod is pullback} and Proposition \ref{monadicity}, we obtain

\begin{theo}\label{main theo flat}
Let $\mathrm{PSh}_{I}(\Cylsf(\Csf_1,\Ccal))$ be the category
of $I$-representable presheaves on the circle category
$\Cylsf(\Csf_1,\Ccal)$. 
There is an equivalence of $\Kbb$-linear categories
\eq{
	\mathrm{PSh}_{I}(\Cylsf(\Csf_1,\Ccal))
\cong \Zsf(\Ccal)\, ,
}
\end{theo}

\begin{rem}\label{extension to all n remark}\hspace{2em}
\begin{enumerate}
	\item 
Since $\Ccal$ is not required to be fusion, the Karoubification of the circle category $\Cylsf(\Csf_1,\Ccal)$ does not, in general, yield the full center
$\Zsf(\Ccal)$.

 Recall that a projective module for a monad is a retract of a free module (cf. \cite[Section~7.3.2]{turaev2017monoidal}). The Karoubification of the Kleisli
 category only yields the subcategory of $\Zsf(\Ccal)$ which has as objects 
 the objects that under the equivalence $T-\Modsf\simeq \Zsf(\Ccal)$
 correspond to projective $T$-modules. This was our motivation to discuss
  a different completion of the Kleisli category as $I$-representable presheaves on
  the Kleisli category in section \ref{Kleisli and flat functors}.

\item
For the general $2$-framed cylinder $\Csf_n$,  the $2$-framing forces us
to add sufficiently many evaluations and coevaluations so that we
get an equivalence
\eq{\mathrm{PSh}_{I}(\Cylsf(\Csf_n,\Ccal))\simeq \Zsf_n(\Ccal)\,\,.} 
The proof of this is in complete analogy to the case of $\Csf_1$.  
\end{enumerate}	
\end{rem}

Our computation of circle categories for string-nets on framed cylinders $\Csf_n$ is in complete accordance with the results of \cite[Corollary~3.2.3, Table~3]{douglas2020dualizable}. 
\bibliographystyle{alpha}
\bibliography{bibliography_arx}

\end{document}